\documentclass[letterpaper,11pt,leqno]{amsart}
\usepackage[latin1]{inputenc}
\usepackage{latexsym}
\usepackage[dvips]{graphicx}
\usepackage{amsfonts}
\usepackage{epsfig}
\usepackage{geometry}
\usepackage{color}
\usepackage{amsmath}
\usepackage{amssymb}

\newtheorem{theorem}{Theorem}[section]
\newtheorem{lemma}[theorem]{Lemma}
\newtheorem{corollary}[theorem]{Corollary}
\newtheorem{remark}[theorem]{Remark}
\newtheorem{proposition}[theorem]{Proposition}

\newtheorem{definition}[theorem]{Definition}

\begin{document}

\title{Dynamical systems and forward-backward algorithms associated with the sum of a convex subdifferential \\ and a monotone cocoercive operator}

\author{Boushra Abbas }

\author{H\'edy Attouch}

\address{Institut de Math\'ematiques et Mod\'elisation de Montpellier, UMR 5149 CNRS, Universit\'e Montpellier 2, place Eug\`ene Bataillon,
34095 Montpellier cedex 5, France}
\email{hedy.attouch@univ-montp2.fr, }

\vspace{0.5cm}

\date{March 24, 2014}

\begin{abstract}

In a Hilbert framework, we introduce  continuous and discrete dynamical systems which
aim at solving inclusions governed by structured
monotone operators $A=\partial\Phi+B$, where $\partial\Phi$ is the
subdifferential of a convex lower semicontinuous function $\Phi$,
and $B$ is a monotone cocoercive operator. We first consider the extension to this setting of the  regularized Newton dynamic
with two potentials which was considered in \cite {AAS}. 
Then, we revisit some related dynamical systems, namely the semigroup of contractions generated by $A$, and the continuous gradient projection dynamic of \cite{Bo}.
By a Lyapunov analysis, we show the convergence properties of the orbits of these systems. 
 The time discretization of these dynamics gives various  forward-backward splitting methods (some new) for solving structured monotone inclusions
involving non-potential terms. The convergence of these algorithms is obtained under classical step size limitation.
Perspectives are given in the field of numerical splitting methods for optimization, and multi-criteria decision processes. 
\end{abstract}

\thanks{\textsf{Contact author: H. Attouch, }\texttt{hedy.attouch@univ-montp2.fr}\textsf{,
phone: +33 04 67 14 35 72, fax: +33 04 67 14 35 58.}%
}

\maketitle

\vspace{0.5cm}


\vspace{1cm}


\markboth{B. ABBAS, H. ATTOUCH}
  {Dynamics and FB algorithms with cocoercive operators}

\section*{Introduction}

Throughout this paper, $\mathcal H$ is a real Hilbert space with scalar product
$\left\langle .,.\right\rangle $ and norm $\|\cdot{}\|$. 
We are going to study some  continuous and discrete 
 dynamics which aim at solving structured monotone inclusions
of the following type 
\begin{equation}
\partial\Phi(x)+Bx\ni0\label{basicpb}
\end{equation}
 where $\partial\Phi$ is the subdifferential of a convex lower semicontinuous
function $\Phi: \mathcal H\rightarrow\mathbb{R}\cup\left\{ +\infty\right\} $,
and $B$ is a monotone cocoercive operator. Recall that a monotone
operator $B: \mathcal H\rightarrow \mathcal H$ is cocoercive if there exists a constant
$\beta>0$ such that for all $x,y\in \mathcal H$ 
\[
\left\langle Bx-By,x-y\right\rangle \geq\beta\left\Vert Bx-By\right\Vert ^{2}.
\]
 The abstract formulation (\ref{basicpb}) covers a large variety
of problems in physical and decision sciences, see for example \cite{abc}, \cite{AM}, \cite{BC},
\cite{Zhud96}, and the discussion at the end of the paper. It is directly connected to two important areas, namely
convex optimization (take $B=0$), and the theory of fixed point for
nonexpansive mappings (take $\Phi=0$, and $B=I-T$ with $T$ a nonexpansive
mapping). It comes naturally into play when we consider both aspects within a physical or decision process.

By a classical result, the two operators $\partial\Phi$
and $B$ are maximal monotone, as well as their sum $A=\partial\Phi+B$.
We will exploit the structure of the maximal monotone operator
$A$, first  to develop  continuous dynamics, and then, by  time
discretization, splitting forward-backward algorithms that aim to
solve (\ref{basicpb}). As a common characteristic of these dynamics, they are first-order evolution equations, whose stationary points are precisely the zeroes of the operator $A=\partial\Phi+B$.
 Among these dynamics some are new, and for others it is an opportunity to revisit, and extend some convergence results with a unifying perspective. 

1. Our first  concern  is the Newton-like dynamic approach to solving monotone
inclusions which was introduced in \cite{AS}. To adapt it to structured
monotone inclusions and splitting methods, this study was developed
in \cite{AAS}, where the operator is the sum of the
subdifferential of a convex lower semicontinuous function, and the
gradient of a convex differentiable function. We wish to extend this
study to a non potential case, and so enlarge its range of applications.
Specifically, our analysis focuses on the convergence properties
(as $t\to+\infty$) of the orbits of the system (\ref{basic-system1})-(\ref{basic-system2})
\begin{align}
&\upsilon\left(t\right)\in\partial\Phi\left(x\left(t\right)\right)\label{basic-system1}\\
&\lambda\dot{x}\left(t\right)+\dot{\upsilon}\left(t\right)+\upsilon\left(t\right)+B\left(x\left(t\right)\right)=0.\label{basic-system2}
\end{align}
 In (\ref{basic-system2}), $\lambda$ is a positive constant which
acts as a Levenberg-Marquard regularization parameter. 
When $\lambda$ is small,  and $B=0$,  the system is close to the continuous Newton method
for solving $\partial\Phi(x)\ni0$.
 The $x$ components of the stationary points of the $(x,v)$ system (\ref{basic-system1})-(\ref{basic-system2})
 are precisely the zeroes of the operator $A=\partial\Phi+B$.
The Cauchy problem for (\ref{basic-system1})-(\ref{basic-system2})
is well-posed. Indeed, by introducing the new unknown function $y(\cdot)=x(\cdot)+\mu v(\cdot)$,
and setting $\mu=\frac{1}{\lambda}$, (\ref{basic-system1})-(\ref{basic-system2})
can be equivalent written as 
\[
\begin{cases}
x(t)=\mbox{prox}_{\mu\Phi}(y(t)),\\
\dot{y}(t)+y(t)-\mbox{prox}_{\mu\Phi}(y(t))+\mu B\left(\mbox{prox}_{\mu\Phi}(y(t))\right)=0,
\end{cases}
\]
 where $\mbox{prox}_{\mu\Phi}$ is the proximal mapping of $\mu\Phi$. Since $\mbox{prox}_{\mu\Phi}$
and $B$ are Lipschitz continuous operators, the above differential
equation (with respect to $y$) is relevant to Cauchy-Lipschitz theorem.
 Under the sole assumption that the solution set $S$ of (\ref{basicpb}) is not empty, in Theorem \ref{conv1}
 we will show that, for
any orbit of system (\ref{basic-system1})-(\ref{basic-system2}), $x(\cdot)$ converges weakly to an element of $S$. 
Strong convergence is obtained under the assumption $\Phi$ inf-compact, or strongly convex.

The above system is regular with respect to the new variable $y$. Its   explicit discretization  gives (with a constant step size $h>0$),
the following algorithm: \ $(x_{k},y_{k})\rightarrow(x_{k},y_{k+1})\rightarrow(x_{k+1},y_{k+1})$,
\[
{\rm (FBN)}\ \begin{cases}
x_{k}=\mbox{prox}_{\mu\Phi}(y_{k}),\\
y_{k+1}=(1-h)y_{k}+h\left(x_{k}-\mu B\left(x_{k}\right)\right).
\end{cases}
\]
We will show in Theorem \ref{T:main} that, under the assumption,   $0<h\leq1$, and $0<\mu<2\beta$,
(FBN) generates sequences that converge weakly to equilibria. Indeed, this is a limitation of the step size very similar that of the classical forward-backward algorithm.
 Note that, when $h \neq 1$, the algorithm (FBN) differs from the classical forward-backward algorithm, by order in the composition of the two basic blocks $\mbox{prox}_{\mu\Phi}$ and 
 $ I - \mu B$. 
 
\smallskip

2. Then, we consider a
naturally related dynamical system, which is the semigroup of contractions
generated by $-A$, $A=\partial\Phi+B$, whose orbits
are the solution trajectories of the differential inclusion 
\begin{equation}
\dot{x}\left(t\right)+\partial\Phi(x(t))+B\left(x\left(t\right)\right)\ni0.\label{sg00}
\end{equation}
In Theorem \ref{sg2},   we show the weak convergence of the orbits of  (\ref{sg00}) to solutions of (\ref{basicpb}), a property which surprisingly has not been systematically
studied before. 
Explicit-implicit time discretization of (\ref{sg00}) gives the classical forward-backward algorithm.

\smallskip

3. Finally, we consider the  dynamic which is associated to the  reformulation of  (\ref{basicpb}) as a fixed point problem:
\begin{equation}
\dot{x}\left(t\right)+   x(t) -  \mbox{prox}_{\mu\Phi} \left(  x(t) - \mu B\left(x\left(t\right) \right)\right)  = 0.\label{gp000}
\end{equation}
It is a regular dynamic which is relevant of Cauchy-Lipschitz theorem. Its convergence properties have been first investigated by Antipin \cite{Ant} and Bolte \cite{Bo} in the 
particular case where $\Phi$ is the indicator function of a closed convex set $C$, and $B$ is the gradient of a convex differentiable function. In that case, the above system specializes to the continuous gradient projection method.   In Theorem \ref{pg-Thm3}  we extend these convergence results to our general setting.
The  explicit time discretization of (\ref{gp000})  gives the relaxed forward-backward algorithm
\begin{center}
 $x_{k+1} =  (1- h)x_k +  h \mbox{prox}_{\mu\Phi}\left( x_k  - \mu B(x_{k})\right).$
\end{center}

A thorough comparative study of forward-backward algorithms provided by discretization of these various related systems is an important issue from a numerical point of view.
It is a subject of ongoing study (see \cite {APR}), which is beyond the scope of this document.

\smallskip

The paper is organized as follows: In Section \ref{RN},  we study the convergence properties of the orbits of the continuous
dynamical system (\ref{basic-system1})-(\ref{basic-system2}).
In Section \ref{algo},  we 
 show the convergence properties of the 
forward-backward (FBN) algorithm which is obtained by time discretization of  (\ref{basic-system1})-(\ref{basic-system2}). In Section \ref{sgroup}, we examine the convergence
properties of the orbits of the semigroup generated by $-(\partial\Phi+B)$,
and make the link with the  
 classical  FB algorithm. 
In Section \ref{pg. section}, we introduce the proximal-gradient dynamical system, study its convergence properties, 
and make the link with the  
 relaxed FB algorithm.
We complete this study by some perspectives  in the realm of numerical optimization, and multi-criteria decision processes.

\section{The continuous regularized Newton-like dynamic}\label{RN}

\subsection{Definition, global existence}\label{def-glob-exist}
By applying the Minty transformation to $\partial\Phi$, system (\ref{basic-system1})-(\ref{basic-system2})
can be reformulated in a form which is relevant to the Cauchy-Lipschitz
theorem, see \cite{AAS}, \cite{ARS}, \cite{AS}. First set $\mu=\frac{1}{\lambda}$
and rewrite (\ref{basic-system2}) as 
\begin{equation}
\dot{x}\left(t\right)+\mu\dot{\upsilon}\left(t\right)+\mu\upsilon\left(t\right)+\mu B\left(x\left(t\right)\right)=0.\label{bas-mu}
\end{equation}
Let us introduce the new unknown function $y(\cdot)=x(\cdot)+\mu v(\cdot)$. Since $v(\cdot) \in \partial \Phi (x(\cdot))$,    we have
$x(\cdot) = \mbox{prox}_{\mu\Phi}y(\cdot)$, where $\mbox{prox}_{\mu\Phi}$ is   the proximal mapping associated to $\mu\Phi$. Recall that $\mbox{prox}_{\mu\Phi}= (I+\mu\partial\Phi)^{-1}$, where
$(I+\mu\partial\Phi)^{-1}$ is the resolvent of index $\mu>0$ of the maximal monotone operator $\partial\Phi$. 
We  obtain the
equivalent dynamic 
\begin{equation}
\begin{cases}
x\left(t\right)=\mbox{prox}_{\mu\Phi}(y\left(t\right))\\
\dot{y}\left(t\right)+y\left(t\right)-\mbox{prox}_{\mu\Phi}(y\left(t\right))+\mu B\left(\mbox{prox}_{\mu\Phi}(y\left(t\right))\right)=0,
\end{cases}\label{bas3}
\end{equation}
 which makes use only of the proximal mapping associated to $\mu\Phi$, and $B$,
which are both Lipschitz continuous operators. Indeed, for any $\mu >0$, the operator $\mbox{prox}_{\mu\Phi}$  is firmly  nonexpansive, see \cite[Proposition 12.27]{BC}. When $\Phi$ is equal
to the indicator function of a closed convex set $C\subset \mathcal H$, $\mbox{prox}_{\mu\Phi}$ is independent of $\mu$, and is equal to $\mbox{proj}_{C}$,
the projection operator on $C$ (whence the proximal terminology, introduced by Moreau).

By Lemma \ref{coco-lip} below,  $B$ is a maximal
monotone Lipschitz continuous operator. Thus, by specializing Theorem
3.1. of \cite{AAS} to our situation, we obtain that the Cauchy problem
for (\ref{basic-system1})-(\ref{basic-system2}) is well-posed. More
precisely, 
\begin{theorem}\label{basic-exist} Let $\lambda>0$ be a positive
constant. Suppose that $\partial\Phi$ is the subdifferential of a
convex lower semicontinuous proper function $\Phi: \mathcal H\rightarrow\mathbb{R}\cup\left\{ +\infty\right\} $,
and that $B:\mathcal H\rightarrow \mathcal H$ is a cocoercive operator on $\mathcal H.$ Let
$\left(x_{0},\upsilon_{0}\right)\in \mathcal H\times \mathcal H$ be such that $\upsilon_{0}\in\partial\Phi\left(x_{0}\right)$.\\
 Then, there exists a unique strong global solution $\left(x\left(\cdot\right),\upsilon\left(\cdot\right)\right):\left[0,+\infty\right[\rightarrow \mathcal H\times \mathcal H$
of the Cauchy problem 
\begin{align}
 & \upsilon\left(t\right)\in\partial\Phi\left(x\left(t\right)\right);\label{2}\\
 & \lambda\dot{x}\left(t\right)+\dot{\upsilon}\left(t\right)+\upsilon\left(t\right)+B\left(x\left(t\right)\right)=0;\label{3}\\
 & x\left(0\right)=x_{0},\upsilon\left(0\right)=\upsilon_{0}.\label{4}
\end{align}
 \end{theorem} In the above statement, we use the following notion of
strong solution, as defined in \cite{AAS}, and \cite{AS}. 

\begin{definition}
We say that the pair $\left(x\left(\cdot\right),\upsilon\left(\cdot\right)\right)$
is a strong global solution of {\rm(\ref{2})-(\ref{3})-(\ref{4})} iff
the following properties are satisfied:

\smallskip{}

$\left(i\right)$ $x\left(\cdot\right),\upsilon\left(\cdot\right):[0,+\infty[\rightarrow \mathcal H$
are absolutely continuous on each interval $\left[0,b\right]$, $0<b<+\infty;$

\smallskip{}

$\left(ii\right)$ $\upsilon(t)\in\partial\Phi\left(x\left(t\right)\right)$
\ for all $t\in[0,+\infty[$;

\smallskip{}

$\left(iii\right)$ $\lambda\dot{x}\left(t\right)+\dot{\upsilon}\left(t\right)+\upsilon\left(t\right)+B\left(x\left(t\right)\right)=0$
\ for almost all $t\in[0,+\infty[$; \smallskip{}

$\left(iv\right)$ $x\left(0\right)=x_{0}$, $\upsilon\left(0\right)=\upsilon_{0}$.
\end{definition}

Equivalent systems (\ref{basic-system1})-(\ref{basic-system2}) and
(\ref{bas3}) provide a dynamic 
whose time discretization  yields
a new class of forward-backward algorithms.

\begin{remark} { \rm For sake of simplicity, we have taken the regularization
parameters $\lambda$ and $\mu$ constant. Indeed, the conclusion
of Theorem \ref{basic-exist} still holds true, just assuming that
$\lambda:[0,+\infty[\to]0,+\infty[$ is absolutely continuous on each
bounded interval $[0,b], \ 0<b<+\infty$ (indeed, it is enough assuming
that $\lambda$ is locally of bounded variation). Taking $\lambda$
varying and asymptotically vanishing provides
a dynamic which is asymptotically close the Newton dynamic
associated to $\Phi$, see \cite{AAS}, \cite{ARS}, \cite{AS}.
This is an important issue for fast converging methods, a subject for further studies.}  
\end{remark}

\subsection{Cocoercive operators}\label{cocoer}
We collect some facts that will be useful.
\begin{lemma}\label{coco-lip} 
Let $B:\mathcal H\to \mathcal H$ be a $\beta$-cocoercive operator. Then, $B$ is $\frac{1}{\beta}-$Lipschitz continuous. 
\end{lemma} 
\begin{proof} Let $x,y\in \mathcal H$. Since
$B$ is $\beta-$cocoercive, by Cauchy-Schwarz inequality we have
\begin{align*}
\beta\left\Vert Bx-By\right\Vert ^{2} & \leq\left\langle Bx-By,x-y\right\rangle \\
 & \leq\left\Vert Bx-By\right\Vert \left\Vert x-y\right\Vert .
\end{align*}
 Hence 
\[
\left\Vert Bx-By\right\Vert \leq\frac{1}{\beta}\left\Vert x-y\right\Vert ,
\]
 which expresses that $B$ is $\frac{1}{\beta}$- Lipschitz continuous.
\end{proof} Let us list some important classes of cocoercive operators. 
\begin{itemize}
\item $B=I-T$ where $T: \mathcal H\to \mathcal H$ is a contraction. One can easily verify
that $B$ is $\frac{1}{2}$-cocoercive.
\item $B=M_{\lambda}$ where $M_{\lambda}$ (with parameter $\lambda>0$)
is the Yosida approximation of a general maximal monotone operator
$M:\mathcal H\to2^{\mathcal H}$, (see \cite{Br}). One can easily verify that $M_{\lambda}$
is $\lambda$-cocoercive. 
\end{itemize}
\noindent By Lemma \ref{coco-lip}, if $B$ is $\beta$-cocoercive,
then it is ${\beta}^{-1}$-Lipschitz continuous. The next lemma, which
provides a converse implication, supplies us with another important
instance of cocoercive operator.
\begin{lemma} {\rm \cite[Corollaire~10]{Bail77}} \label{l:BH}
Let $\Psi: \mathcal H\to\mathbb{R}$ be a differentiable convex function and
let $\tau>0$. Suppose that $\nabla\Psi$ is $\tau$-Lipschitz continuous.
Then $\nabla\Psi$ is $\tau^{-1}$-cocoercive. 
\end{lemma}
\begin{lemma}{\rm \cite[Lemma~2.3]{Comb04}}\label{l:9} Let
$B:\mathcal H\to \mathcal H$ be a $\beta$-cocoercive operator, and let $\mu\in\left]0,2\beta\right[$.
Then $Id-\mu B$ is nonexpansive. 
\end{lemma}
Because of the cocoercive property of $B$, its inverse operator $B^{-1}$
is strongly monotone. Hence, even if the primal problem (\ref{basicpb})
has multiple solutions, the Attouch-Th\'era dual problem (\cite{AT})
\[
B^{-1}\xi-\partial\Phi^{*}(-\xi)\ni0
\]
 has a unique solution (with $\xi=Bz$, and $z$ solution of the primal
problem). Returning to the primal problem (\ref{basicpb}), this gives
the following result (we give below another direct proof which does
not use a duality argument).

\begin{lemma}\label{coco-2} 
Let $B$ be a maximal monotone operator
which is cocoercive, and let $\partial\Phi$ be the subdifferential
of a convex lower semicontinuous proper function $\Phi:\mathcal H\rightarrow\mathbb{R}\cup\left\{ +\infty\right\} $.
Set $S=\left\{ z\in \mathcal H;\;\partial\Phi(z)+Bz\ni0\right\} $ be the solution
set of {\rm(\ref{basicpb})}. Then $Bz$ is a constant vector, as $z$
varies over $S$.
 \end{lemma} 
\begin{proof} Let $z_{1}$ and $z_{2}$
be two elements of $S$. Hence $-Bz_{1}\in\partial\Phi(z_{1})$ and
$-Bz_{2}\in\partial\Phi(z_{2})$. By the monotonicity property of
$\partial\Phi$ 
\[
\left\langle -Bz_{1}+Bz_{2},z_{1}-z_{2}\right\rangle \geq0.
\]
 Equivalently 
\[
0\geq\left\langle Bz_{1}-Bz_{2},z_{1}-z_{2}\right\rangle .
\]
 Combining this inequality with the cocoercive property of $B$, 
 $$\left\langle Bz_{1}-Bz_{2},z_{1}-z_{2}\right\rangle \geq\beta\left\Vert Bz_{1}-Bz_{2}\right\Vert ^{2},$$
we obtain $0\geq\beta\left\Vert Bz_{1}-Bz_{2}\right\Vert ^{2}$, that
is $Bz_{1}=Bz_{2}$. 
\end{proof}

\subsection{Convergence of the regularized Newton-like dynamic}\label{asympt-conv}

We will study the convergence properties of the orbits of system (\ref{2})-(\ref{3}),
whose existence is guaranteed by Theorem \ref{basic-exist}. We call
\[
S=\left\{ z\in \mathcal H;\;\partial\Phi(z)+Bz\ni0\right\} 
\]
 the solution set of problem (\ref{basicpb}), and we assume that
$S\neq\emptyset$.
 Let $\left(x\left(\cdot\right),\upsilon\left(\cdot\right)\right): [0,+\infty [ \rightarrow \mathcal H \times \mathcal H$
be the solution of the Cauchy problem (\ref{2})-(\ref{3})-(\ref{4}). Equivalently with $\mu = \frac{1}{\lambda}$
\begin{align}
 & \upsilon\left(t\right)\in\partial\Phi\left(x\left(t\right)\right);\label{200}\\
 & \dot{x}\left(t\right)+ \mu \dot{\upsilon}\left(t\right)+ \mu  \upsilon\left(t\right)+ \mu  B\left(x\left(t\right)\right)=0;\label{300}\\
 & x\left(0\right)=x_{0},\upsilon\left(0\right)=\upsilon_{0}.\label{400}
\end{align}
 We will use the following  functions: for any $z\in S$,
for any $t\geq0$ 
\begin{align*}
 & g_{z}\left(t\right):=\Phi\left(z\right)-\left[\Phi\left(x\left(t\right)\right)+\left\langle z-x\left(t\right),\upsilon\left(t\right)\right\rangle \right]\\
 & \Gamma_{z}\left(t\right):=\frac{1}{2}\left\Vert x\left(t\right)-z\right\Vert ^{2}+\mu g_{z}\left(t\right).
\end{align*}
Note that $\Gamma_{z}(t)$ is a Bregman distance between $x(t)$ and $z$. It is associated with the convex function $x \mapsto \frac{1}{2}\|x\|^2 + \mu\Phi (x)$.
In our nonsmooth setting, it  combines the metric of $\mathcal H$ with the metric associated to the \textquotedblleft Hessian\textquotedblright \ of $\Phi$.
Our proof of the convergence is based on Lyapunov analysis, and the fact that $t \mapsto \Gamma_{z}(t)$ is a decreasing function.
Let us state our main convergence result.

\begin{theorem}\label{conv1} Suppose  $S\neq\emptyset$. Then for
all  $x(\cdot)$ orbit of the system {\rm(\ref{2})-(\ref{3})},  the following convergence properties are satisfied, when t tends to infinity:

\smallskip{}

1. \ $\underset{t\longrightarrow+\infty}{\lim}\left\Vert \upsilon\left(t\right)+B\left(x\left(t\right)\right)\right\Vert =0$;

\smallskip{}

2. \ $B(x(\cdot))$ converges strongly to $Bz$, where $Bz$ is uniquely
defined for $z\in S$.

\smallskip{}

3. \ $\upsilon\left(\cdot \right)$ converges strongly to $-Bz$, where
$Bz$ is uniquely defined for $z\in S$.

\smallskip{}

4. \ $x(\cdot)$ converges weakly to an element of $S$.

\end{theorem}

The proof of Theorem  \ref{conv1} has been extended to the end of this section. We collect first few preliminary technical lemma, then we conduct a Lyapunov-type analysis, and finally prove Theorem
\ref{conv1} and some convergence results which are connected.

\smallskip

\noindent \textbf{A. Preliminary results} \
We will frequently use the
following derivation chain rule for a convex lower semicontinuous
function $\Phi: \mathcal H\to\mathbb{R}\cup\left\{ +\infty\right\} $, see \cite[Lemma 3.3]{Br}.

\begin{lemma}\label{deriv-chain} Suppose that the assumptions i), ii), iii) are satisfied:

$i)$ $v(t)\in\partial\Phi(x(t))$ for almost every $t\in[0,b]$;

$ii)$ $v$ belongs to $L^{2}(0,b; \mathcal H)$;

$iii)$ $\dot{x}\in L^{2}(0,b; \mathcal H)$ .

\noindent Then, $t\mapsto\Phi\left(x\left(t\right)\right)$ is absolutely continuous
on $\left[0,b\right]$, and, for almost every $t\in[0,b]$, 
\begin{equation}
\frac{d}{dt}\Phi\left(x\left(t\right)\right)=\left\langle \upsilon\left(t\right),\dot{x}\left(t\right)\right\rangle .\label{chain2}
\end{equation}
 \end{lemma}

\noindent In order to prove the weak convergence of the trajectories
of system (\ref{2})-(\ref{3}), we will use the  Opial's lemma
\cite{Op} that we recall in its continuous form; see also \cite{Bruck},
who initiated the use of this argument to analyze the asymptotic convergence
of nonlinear contraction semigroups in Hilbert spaces.

\begin{lemma}\label{Opial} Let $S$ be a non empty subset of $\mathcal H$
and $x:[0,+\infty[\to \mathcal H$ a map. Assume that 
\begin{eqnarray*}
(i) &  & \mbox{for every }z\in S,\>\lim_{t\to+\infty}\|x(t)-z\|\mbox{ exists};\\
(ii) &  & \mbox{every weak sequential cluster point of the map }x\mbox{ belongs to }S.
\end{eqnarray*}
 Then 
\[
w-\lim_{t\to+\infty}x(t)=x_{\infty}\ \ \mbox{ exists, for some element }x_{\infty}\in S.
\]
 \end{lemma}
We will also need the
following lemma from \cite{AAS}. 
\begin{lemma} \label{lm:aux} Suppose that
$1\leq p<\infty$, $1\leq r\leq\infty$, $F\in L^{p}([0,\infty[)$
is a locally absolutely continuous nonnegative function, $G\in L^{r}([0,\infty[)$
and for almost all $t$ 
\begin{equation}
\frac{d}{dt}F(t)\leq G(t).\label{eq:lm2}
\end{equation}
 Then $\lim_{t\to\infty}F(t)=0$. 
 \end{lemma}

\smallskip

\noindent \textbf{B. Lyapunov analysis} \
As a main ingredient of our convergence proof, we are going to show that $\Gamma_{z}$
is a strict Lyapunov function. More precisely, 
\begin{proposition}\label{estim-x}
Suppose that $S\neq\emptyset$. Then, for any $z\in S,$ $\Gamma_{z}$
is a decreasing nonnegative function, and hence  $\lim_{t\rightarrow+\infty}\Gamma_{z}(t)$ exists. Moreover

\vspace{1mm}

$1.\ \left\Vert B\left(x\right)-B\left(z\right)\right\Vert \in L^{2}\left(\left[0,+\infty\right[\right)$;

\vspace{1mm}

$2.\ \ x\ \mbox{is bounded}$;

\vspace{1mm}

$3.\ \left\Vert \dot{x}\right\Vert \in L^{2}\left(\left[0,+\infty\right[\right)$;

\vspace{1mm}

$4.\ \left\Vert \dot{\upsilon}+\upsilon+B\left(z\right)\right\Vert \in L^{2}\left(\left[0,+\infty\right[\right)$;

\vspace{1mm}

$5.\ \ \left\Vert \dot{\upsilon}\right\Vert \in L^{2}\left(\left[0,+\infty\right[\right)$;

\vspace{1mm}

$6.\ \ \left\Vert v\right\Vert \in L^{\infty}\left(\left[0,+\infty\right[\right)$.

\vspace{1mm}

$7.\ \lim_{t\rightarrow+\infty}\left(\Phi\left(x\left(t\right)\right)+\left\langle x\left(t\right),B\left(z\right)\right\rangle \right)\ \ \mbox{exists}$.
\end{proposition} 

In order to prove Proposition \ref{estim-x}, let us first establish some technical results.

 \begin{lemma}\label{g-prop} For any $t\geq0$ and $z\in S$ 
\[
g_{z}\left(t\right)\geq0
\]
 and for almost all $t\geq0$ 
\[
\frac{d}{dt}g_{z}\left(t\right)=\left\langle x\left(t\right)-z,\dot{\upsilon}\left(t\right)\right\rangle. 
\]
 \end{lemma} 
 \begin{proof} The first inequality $g_{z}\left(t\right)\geq0$,
follows from the subdifferential inequality for $\Phi$ at $x\left(t\right)$,
and $\upsilon\left(t\right)\in\partial\Phi\left(x\left(t\right)\right)$.
By the derivation chain rule in the nonsmooth convex case, see Lemma \ref{deriv-chain}, and $\upsilon(t)\in\partial\Phi\left(x(t)\right)$,
we have $\frac{d}{dt}\Phi\left(x\left(t\right)\right)=\left\langle \upsilon\left(t\right),\dot{x}\left(t\right)\right\rangle $.
Hence 
\begin{align*}
\frac{d}{dt}g_{z}\left(t\right) & =-\frac{d}{dt}\Phi\left(x\left(t\right)\right)+\left\langle \dot{x}\left(t\right),\upsilon\left(t\right)\right\rangle +\left\langle x\left(t\right)-z,\dot{\upsilon}\left(t\right)\right\rangle \\
 & =-\left\langle \dot{x}\left(t\right),\upsilon\left(t\right)\right\rangle +\left\langle \dot{x}\left(t\right),\upsilon\left(t\right)\right\rangle +\left\langle x\left(t\right)-z,\dot{\upsilon}\left(t\right)\right\rangle \\
 & =\left\langle x\left(t\right)-z,\dot{\upsilon}\left(t\right)\right\rangle .
\end{align*}
 \end{proof} 
 The following result from \cite{AS} will also be
useful. 
\begin{lemma} \label{pr:basic} 
For almost every $t>0$ the
following properties hold: 
\begin{equation}
\left\langle \dot{x}(t),\dot{v}(t)\right\rangle \geq0.\label{eq:ineq.bas}
\end{equation}
 \end{lemma} 
 \begin{proof} For almost every $t>0,$ \ $\dot{x}(t)$
and $\dot{v}(t)$ are well defined, thus 
\[
\left\langle \dot{x}(t),\dot{v}(t)\right\rangle =\lim_{h\to0}\ \frac{1}{h^{2}}\left\langle x(t+h)-x(t),v(t+h)-v(t)\right\rangle .
\]
 By equation (\ref{3}), we have $v(t)\in\partial\Phi(x(t)$. Since
$\partial\Phi: \mathcal H\to2^{\mathcal H}$ is monotone 
\[
\left\langle x(t+h)-x(t),v(t+h)-v(t)\right\rangle \geq0.
\]
 Dividing by $h^{2}$, and passing to the limit preserves the inequality,
which yields \eqref{eq:ineq.bas}. 
\end{proof}

We can now proceed with the proof of Proposition \ref{estim-x}.
 
\begin{proof} By definition of $\Gamma_{z}$, and Lemma \ref{g-prop} 
\begin{align}
\frac{d}{dt}\Gamma_{z}\left(t\right) & =\left\langle \dot{x}\left(t\right),x\left(t\right)-z\right\rangle +\mu\left\langle \dot{\upsilon}(t),x\left(t\right)-z\right\rangle \nonumber \\
 & =\left\langle x\left(t\right)-z,\dot{x}(t)+\mu\dot{\upsilon}(t)\right\rangle. \label{5}
\end{align}
 From (\ref{300}) and (\ref{5}) we deduce that 
\begin{equation}
\frac{d}{dt}\Gamma_{z}\left(t\right)+\mu\left\langle x\left(t\right)-z,\upsilon(t)+B\left(x\left(t\right)\right)\right\rangle =0.\label{gamma1}
\end{equation}
 Since $z\in S$, we have $\partial\Phi\left(z\right)+B\left(z\right)\ni0$.
Equivalently, there exists some $\xi\in\partial\Phi(z)$ such that
$\xi+Bz=0$. By monotonicity of $\partial\Phi$, and $\upsilon(t)\in\partial\Phi(x(t))$
we have 
\begin{equation}
\left\langle x\left(t\right)-z,\upsilon(t)-\xi\right\rangle \geq0.\label{gamma2}
\end{equation}
 Let us rewrite (\ref{gamma1}) as 
\[
\frac{d}{dt}\Gamma_{z}\left(t\right)+\mu\left\langle x(t)-z,\upsilon(t)-\xi\right\rangle +\mu\left\langle x(t)-z,\xi+B(x(t))\right\rangle =0,
\]
 which, from (\ref{gamma2}), gives 
\[
\frac{d}{dt}\Gamma_{z}\left(t\right)+\mu\left\langle x(t)-z,\xi+B(x(t))\right\rangle \leq0.
\]
 From $\xi+Bz=0$, we deduce that 
\[
\frac{d}{dt}\Gamma_{z}\left(t\right)+\mu\left\langle x(t)-z,B(x(t))-Bz\right\rangle \leq0.
\]
 By the cocoercive property of $B$ we infer 
\begin{equation}
\frac{d}{dt}\Gamma_{z}\left(t\right)+\mu\beta\left\Vert B\left(x\left(t\right)\right)-B\left(z\right)\right\Vert ^{2}\leq0.\label{gamma3}
\end{equation}
 From (\ref{gamma3}), we readily obtain that $\Gamma_{z}$ is a decreasing
function. Being nonnegative, it converges to a finite value. By integration
of the above inequality, and using that $\Gamma_{z}$ is nonnegative
we obtain 
\[
\int_{0}^{\infty}{\left\Vert B\left(x(t)\right)-B\left(z\right)\right\Vert }^{2}dt<+\infty,
\]
 that's item 1. Since $g_{z}$ is nonnegative, and $\Gamma_{z}$ is
bounded from above, we deduce from the definition of $\Gamma_{z}$
that $\left\Vert x\left(t\right)-z\right\Vert ^{2}$ is bounded, which
implies that the orbit $x$ is bounded, that's item 2.\\
 To prove item 3., we return to (\ref{gamma3}), and combine it with
(\ref{3}), $\lambda\dot{x}\left(t\right)+\dot{\upsilon}\left(t\right)+\upsilon\left(t\right)+B\left(x\left(t\right)\right)=0$,
to obtain 
\[
\frac{d}{dt}\Gamma_{z}\left(t\right)+\mu\beta\left\Vert \lambda\dot{x}\left(t\right)+\dot{\upsilon}\left(t\right)+\upsilon\left(t\right)+B\left(z\right)\right\Vert ^{2}\leq0.
\]
 After developing, we obtain 
\begin{equation}
\frac{d}{dt}\Gamma_{z}\left(t\right)+\mu\beta\lambda^{2}\left\Vert \dot{x}\left(t\right)\right\Vert ^{2}+\mu\beta\left\Vert \dot{\upsilon}\left(t\right)+\upsilon\left(t\right)+B\left(z\right)\right\Vert ^{2}+2\mu\beta\lambda\left\langle \dot{x}\left(t\right),\dot{\upsilon}\left(t\right)+\upsilon\left(t\right)+B\left(z\right)\right\rangle \leq0.\label{gamma4}
\end{equation}
 Examine the last term of the left member. We have

\smallskip{}

$\left\langle \dot{x}\left(t\right),\dot{\upsilon}\left(t\right)\right\rangle \geq0$
\
by Lemma \ref{pr:basic}.

\smallskip{}

$\left\langle \dot{x}\left(t\right),\upsilon\left(t\right)\right\rangle =\frac{d}{dt}\Phi\left(x\left(t\right)\right)$
by Lemma \ref{deriv-chain}.

\smallskip{}

$\left\langle \dot{x}\left(t\right),B\left(z\right)\right\rangle =\frac{d}{dt}\left\langle x\left(t\right),B\left(z\right)\right\rangle $.

\smallskip{}

\noindent Combining (\ref{gamma4}), $\mu \lambda =1$,  and the above formulas we obtain
\begin{equation}
\frac{d}{dt}\left[\Gamma_{z}\left(t\right)+2\beta\left(\Phi\left(x\left(t\right)\right)+\left\langle x\left(t\right),B\left(z\right)\right\rangle \right)\right]+\beta\lambda\left\Vert \dot{x}\left(t\right)\right\Vert ^{2}+\mu\beta\left\Vert \dot{\upsilon}\left(t\right)+\upsilon\left(t\right)+B\left(z\right)\right\Vert ^{2}\leq0.\label{gamma5}
\end{equation}
 From this, we directly obtain that 
\[
G_{z}(t):=\Gamma_{z}\left(t\right)+2\beta\left[\Phi\left(x\left(t\right)\right)+\left\langle x\left(t\right),B\left(z\right)\right\rangle \right]
\]
 is a decreasing function. Since the orbit $x$ is bounded, it follows
that $G_{z}$ is bounded from below (use that $\Gamma_{z}$ is nonnegative,
and $\Phi$ admits a continuous affine minorant). Hence $\lim_{t\rightarrow+\infty}G_{z}(t)$
exists. Since $\lim_{t\rightarrow+\infty}\Gamma_{z}\left(t\right)$
also exists, we infer 
\[
\lim_{t\rightarrow+\infty}\left(\Phi\left(x\left(t\right)\right)+\left\langle x\left(t\right),B\left(z\right)\right\rangle \right)\ \ \mbox{exists}.
\]
 By integration of (\ref{gamma5}), and using that $G_{z}$
is bounded from below, we obtain 
\begin{align}
 & \int_{0}^{+\infty}\left\Vert \dot{x}\left(t\right)\right\Vert ^{2}dt<+\infty\label{7}\\
 & \int_{0}^{+\infty}\left\Vert \dot{\upsilon}\left(t\right)+\upsilon\left(t\right)+B\left(z\right)\right\Vert ^{2}dt<+\infty\label{8}.
\end{align}

Let us now establish estimations on $\dot{\upsilon}$. 
Let us start from (\ref{8}), and develop
it. Equivalently, there exists some positive constant $M$ such that
for any $0<T<\infty$ 
\[
\int_{0}^{T}\left(\left\Vert \dot{\upsilon}\left(t\right)\right\Vert ^{2}+\left\Vert \upsilon\left(t\right)+B\left(z\right)\right\Vert ^{2}+2\left\langle \dot{\upsilon}\left(t\right),\upsilon\left(t\right)+B\left(z\right)\right\rangle \right)dt\leq M.
\]
 From 
\[
2\left\langle \dot{\upsilon}\left(t\right),\upsilon\left(t\right)+B\left(z\right)\right\rangle =2\left\langle \dot{\upsilon}\left(t\right),\upsilon\left(t\right)\right\rangle +2\left\langle \dot{\upsilon}\left(t\right),B\left(z\right)\right\rangle =\frac{d}{dt}\left(\left\Vert \upsilon\left(t\right)\right\Vert ^{2}+2\left\langle \upsilon\left(t\right),B\left(z\right)\right\rangle \right)
\]
 we deduce that 
\[
\int_{0}^{T}\left\Vert \dot{\upsilon}\left(t\right)\right\Vert ^{2}dt+\left\Vert \upsilon\left(T\right)\right\Vert ^{2}+2\left\langle \upsilon\left(T\right),B\left(z\right)\right\rangle \leq\left\Vert \upsilon_{0}\right\Vert ^{2}+2\left\langle \upsilon_{0},B\left(z\right)\right\rangle +M.
\]
 This being valid for any $0<T<\infty$, we immediately obtain 
\[
\int_{0}^{+\infty}\left\Vert \dot{\upsilon}\left(t\right)\right\Vert ^{2}dt<+\infty
\]
 and 
\[
\left\Vert v\right\Vert \in L^{\infty}\left(\left[0,+\infty\right[\right),
\]
which completes the proof of Proposition \ref{estim-x}. 
 \end{proof} 
 
 \smallskip

\noindent \textbf{C. Proof of convergence} \
1. By Proposition \ref{estim-x} item 3 and  5, we have
$\dot{x}\in L^{2}\left(\left[0,+\infty\right[\right)$ and $\dot{\upsilon}\in L^{2}\left(\left[0,+\infty\right[\right)$.
Hence 
\begin{equation}
\lambda\dot{x}\left(\cdot\right)+\dot{\upsilon}\left(\cdot\right)\in L^{2}\left(\left[0,+\infty\right[\right).\label{conv2}
\end{equation}
 By combining (\ref{3}), $\lambda\dot{x}\left(t\right)+\dot{\upsilon}\left(t\right)+\upsilon\left(t\right)+B\left(x\left(t\right)\right)=0$,
with (\ref{conv2}) we obtain 
\begin{equation}
\upsilon+B(x)\in L^{2}\left(\left[0,+\infty\right[\right).\label{conv3}
\end{equation}
 Let us apply Lemma \ref{lm:aux} with $F(t)=\frac{1}{2}\left\Vert \upsilon(t)+B(x(t))\right\Vert ^{2}$.
By (\ref{conv3}) we have $F\in L^{1}([0,+\infty[)$.\\
 Let us show that $\frac{d}{dt}F\in L^{1}([0,+\infty[)$. Indeed, it
follows easily from the Lipschitz property of $B$ that $B(x)$ is
absolutely continuous on any bounded set, and that for almost all
$t>0$ 
\begin{equation}
\left\Vert \frac{d}{dt}B(x(t))\right\Vert \leq\frac{1}{\beta}\left\Vert \dot{x}(t)\right\Vert .\label{conv4}
\end{equation}
 From 
\begin{equation}
\frac{d}{dt}F(t)=\left\langle \upsilon(t)+B\left(x(t)\right),\dot{\upsilon}\left(t\right)+\frac{d}{dt}B(x(t))\right\rangle \label{conv5}
\end{equation}
 by Cauchy-Schwarz inequality and (\ref{conv4}), we deduce that 
\begin{equation}
|\frac{d}{dt}F(t)|\leq\left\Vert \upsilon(t)+B\left(x(t)\right)\right\Vert \left(\left\Vert \dot{\upsilon}\left(t\right)\right\Vert +\frac{1}{\beta}\left\Vert \dot{x}(t)\right\Vert \right).\label{conv6}
\end{equation}
 Since $\upsilon+B(x)$, $\dot{x}$ and $\dot{\upsilon}$ belong to
$L^{2}\left(\left[0,+\infty\right[\right)$, we obtain  $\frac{d}{dt}F\in L^{1}([0,+\infty[)$.
By applying Lemma \ref{lm:aux}, we obtain $\lim_{t\to+\infty}F(t)=0$,
which proves the first item of Theorem \ref{conv1}.

\smallskip{}

2. and 3. Let us apply Lemma \ref{lm:aux} with $F_{1}(t)=\frac{1}{2}\left\Vert B(x(t))-Bz\right\Vert ^{2}$.
By Proposition \ref{estim-x}, item 1, we have $F_{1}\in L^{1}([0,+\infty[)$.
Moreover, by using (\ref{conv4}), and a similar argument as above,
we have 
\begin{equation}
|\frac{d}{dt}F_{1}(t)|\leq\frac{1}{\beta}\left\Vert \dot{x}(t)\right\Vert \times\left\Vert B(x(t))-Bz\right\Vert .\label{conv7}
\end{equation}
 Combining $\left\Vert \dot{x}\right\Vert \in L^{2}\left(\left[0,+\infty\right[\right)$
with $\left\Vert B(x)-Bz\right\Vert \in L^{2}\left(\left[0,+\infty\right[\right)$
we deduce from (\ref{conv7}) that $\frac{d}{dt}F_{1}\in L^{1}([0,+\infty[)$.
By Lemma \ref{lm:aux}, we obtain $\lim F_{1}(t)=0$, which is our
claim. Item 3. is a straight consequence of the two previous items.

\smallskip{}

4. Let us verify that the conditions of Opial's lemma \ref{Opial}
are satisfied, taking $S$ equal to the solution set of (\ref{basicpb}).\\
 (i) Let $\bar{x}$ be a weak sequential cluster point of $x$, i.e.,
$\bar{x}=w-\lim x(t_{n})$ for some sequence $t_{n}\rightarrow+\infty$.
The operator $A=\partial\Phi+B$ is maximal monotone, and hence is
demi-closed. From $\upsilon\left(t_{n}\right)+B\left(x\left(t_{n}\right)\right)\rightarrow0$
strongly, $x(t_{n})\rightharpoonup\bar{x}$ weakly, and $\upsilon\left(t_{n}\right)+B\left(x\left(t_{n}\right)\right)\in A(x\left(t_{n}\right))$,
we deduce that $A(\bar{x})=\partial\Phi(\bar{x})+B(\bar{x})\ni0$,
that is $\bar{x}\in S$.\\
 (ii) Let us recall the definition of \\
 $\Gamma_{z}\left(t\right):=\frac{1}{2}\left\Vert x\left(t\right)-z\right\Vert ^{2}+\mu g_{z}\left(t\right)$\\
 $g_{z}\left(t\right):=\Phi\left(z\right)-\left[\Phi\left(x\left(t\right)\right)+\left\langle z-x\left(t\right),\upsilon\left(t\right)\right\rangle \right]$.
By Proposition \ref{estim-x}, for any $z\in S,$ $\Gamma_{z}$ is
a decreasing nonnegative function, and hence converges. Hence, in
order to prove that $\lim_{t\to+\infty}\|x(t)-z\|$ exists, is equivalent
to prove that $\lim_{t\to+\infty}g_{z}\left(t\right)$ exists. To
prove this, we use Lemma \ref{lm:aux} with $F(t)=g_{z}$. \\
 By Lemma \ref{g-prop}, for almost all $t\geq0$ 
\[
\frac{d}{dt}g_{z}\left(t\right)=\left\langle x\left(t\right)-z,\dot{\upsilon}\left(t\right)\right\rangle. 
\]
 Hence 
\begin{equation}
|\frac{d}{dt}g_{z}(t)|\leq\left\Vert \dot{\upsilon}\left(t\right)\right\Vert \times\left\Vert x\left(t\right)-z\right\Vert .\label{conv8}
\end{equation}
 By Proposition \ref{estim-x} item  5, $\left\Vert \dot{\upsilon}\right\Vert \in L^{2}\left(\left[0,+\infty\right[\right)$,
and by Proposition \ref{estim-x} item 2., $x$ is bounded. Hence
$\frac{d}{dt}g_{z}\in L^{2}\left(\left[0,+\infty\right[\right)$.
\smallskip{}
Let us now prove that $g_{z}\in L^{2}\left(\left[0,+\infty\right[\right)$.
Since $g_{z}$ is nonnegative, we just need to majorize it by a square
integrable function. By the convex subdifferential inequality, and
$-Bz\in\partial\Phi(z)$ we have 
\[
\Phi\left(x\left(t\right)\right)\geq\Phi\left(z\right)-\left\langle Bz,x\left(t\right)-z\right\rangle.
\]
 Equivalently 
\[
\Phi\left(z\right)-\Phi\left(x\left(t\right)\right)\leq\left\langle Bz,x\left(t\right)-z\right\rangle .
\]
 Combining this inequality with the definition of $g_{z}$, we obtain
\[
0\leq g_{z}\left(t\right)\leq\left\langle Bz,x\left(t\right)-z\right\rangle -\left\langle z-x\left(t\right),\upsilon\left(t\right)\right\rangle .
\]
 Equivalently 
\[
0\leq g_{z}\left(t\right)\leq\left\langle Bz+\upsilon\left(t\right),x\left(t\right)-z\right\rangle .
\]
 By Cauchy-Schwarz inequality and the triangle inequality we deduce
that 
\begin{equation}
0\leq g_{z}\left(t\right)\leq\left(\left\Vert Bz-Bx(t)\right\Vert +\left\Vert Bx(t)+\upsilon\left(t\right)\right\Vert \right)\left\Vert x\left(t\right)-z\right\Vert .\label{conv9}
\end{equation}
 By Proposition \ref{estim-x} item 1, we have $\left\Vert B\left(x\right)-B\left(z\right)\right\Vert \in L^{2}\left(0,+\infty\right)$.\\
By (\ref{conv3}), $\upsilon+B(x)\in L^{2}\left(\left[0,+\infty\right[\right)$.
Since $x$ is bounded, from (\ref{conv9}) we obtain $g_{z}\in L^{2}\left(\left[0,+\infty\right[\right)$.
Thus $g_{z}$ and $\frac{d}{dt}g_{z}$ belong to $L^{2}\left(\left[0,+\infty\right[\right)$.
By Lemma \ref{lm:aux} we conclude that $\lim_{t\to+\infty}g_{z}\left(t\right)$
exists. Indeed the limit is equal to zero. 

As a direct consequence of the above proof we have the following result.

\begin{proposition}\label{conv-Phi1} Suppose that $S\neq\emptyset$. Then,
for any $z\in S$, 
\[
\Phi(z)-\Phi(x(t))-\left\langle v(t),z-x(t)\right\rangle \longrightarrow0\quad\mbox{as}\ t\rightarrow+\infty.
\]
 \end{proposition}

Let us complete the above result by the following related convergence
properties of the orbits $x$ of system (\ref{2})-(\ref{3}).

\begin{proposition}\label{conv-Phi2} Suppose that $S\neq\emptyset$. Then,
for any $z\in S$, 
\[
\int_{0}^{+\infty}\Phi\left(x\left(t\right)\right)-\Phi\left(z\right)+\left\langle B\left(z\right),x\left(t\right)-z\right\rangle dt<+\infty,
\]
 and 
\[
\Phi(x(t))-\Phi(z)+\left\langle Bz,x(t)-z\right\rangle \longrightarrow0\quad\mbox{as}\ t\rightarrow+\infty,
\]
 where $Bz$ is the element which is uniquely defined for $z\in S$.
 In particular (take $z=x_{\infty}$), 
 \[
 \Phi(x(t)) \rightarrow \Phi(x_{\infty}) \quad\mbox{as}\ t\rightarrow+\infty,
 \]
 where $x_{\infty}\in S$ is the weak limit of the trajectory $t \mapsto x(t)$.
\end{proposition}

 \begin{proof} Let us return to (\ref{gamma1}) 
\[
\frac{d}{dt}\Gamma_{z}\left(t\right)+\mu\left\langle x\left(t\right)-z,\upsilon(t)\right\rangle +\mu\left\langle x\left(t\right)-z,B\left(x\left(t\right)\right)\right\rangle =0.
\]
 By $\upsilon(t)\in\partial\Phi\left(x\left(t\right)\right)$, we
have the subdifferential inequality 
\[
\Phi\left(z\right)\geq\Phi\left(x\left(t\right)\right)+\left\langle z-x\left(t\right),\upsilon(t)\right\rangle .
\]
 Combining the two above relations yields 
\begin{align*}
\frac{d}{dt}\Gamma_{z}\left(t\right)+\mu\left[\Phi\left(x\left(t\right)\right)-\Phi\left(z\right)\right]+\mu\left\langle B\left(x\left(t\right)\right)-B\left(z\right),x\left(t\right)-z\right\rangle +\mu\left\langle B\left(z\right),x\left(t\right)-z\right\rangle  & \leq0.
\end{align*}
 Since $B$ is $\beta-$ cocoercive, we have $\left\langle B\left(x\left(t\right)\right)-B\left(z\right),x\left(t\right)-z\right\rangle \geq\beta\left\Vert B\left(x\left(t\right)\right)-B\left(z\right)\right\Vert ^{2}$.
Hence 
\[
\frac{d}{dt}\Gamma_{z}\left(t\right)+\mu\left[\Phi\left(x\left(t\right)\right)-\Phi\left(z\right)\right]+\mu\beta\left\Vert B\left(x\left(t\right)\right)-B\left(z\right)\right\Vert ^{2}+\mu\left\langle B\left(z\right),x\left(t\right)-z\right\rangle \leq0.
\]
 As a consequence 
\begin{equation}
\frac{d}{dt}\Gamma_{z}\left(t\right)+\mu\left[\Phi\left(x\left(t\right)\right)-\Phi\left(z\right)+\left\langle B\left(z\right),x\left(t\right)-z\right\rangle \right]\leq0.\label{conv-Phi-1}
\end{equation}
 Since $-Bz\in\partial\Phi(z)$ we have $\Phi\left(x\left(t\right)\right)-\Phi\left(z\right)+\left\langle B\left(z\right),x\left(t\right)-z\right\rangle \geq0$.
By integration of (\ref{conv-Phi-1}), and $\Gamma_{z}$ minorized,
we obtain 
\begin{equation}
\int_{0}^{+\infty}\Phi\left(x\left(t\right)\right)-\Phi\left(z\right)+\left\langle B\left(z\right),x\left(t\right)-z\right\rangle dt<+\infty.\label{conv-Phi-2}
\end{equation}
 Let us apply Lemma \ref{lm:aux} with $F_{2}(t)=\Phi\left(x\left(t\right)\right)-\Phi\left(z\right)+\left\langle B\left(z\right),x\left(t\right)-z\right\rangle $.
By (\ref{conv-Phi-2}), and $F_{2}$ nonnegative, we have $F_{2}\in L^{1}([0,+\infty[)$.
Moreover by using the derivation chain rule in the nonsmooth convex
case, see Lemma \ref{deriv-chain}, and $\upsilon(t)\in\partial\Phi\left(x(t)\right)$,
we have $\frac{d}{dt}\Phi\left(x\left(t\right)\right)=\left\langle \upsilon\left(t\right),\dot{x}\left(t\right)\right\rangle $.
Hence 
\begin{equation}
\frac{d}{dt}F_{2}(t)=\left\langle \dot{x}(t),\upsilon\left(t\right)+B\left(z\right)\right\rangle ,\label{conv-Phi-3}
\end{equation}
 which by Cauchy-Schwarz inequality yields 
\begin{equation}
|\frac{d}{dt}F_{2}(t)|\leq\left\Vert \dot{x}(t)\right\Vert \left(\left\Vert \upsilon\left(t\right)\right\Vert +\left\Vert Bz\right\Vert \right).\label{conv-Phi-4}
\end{equation}
 By Proposition \ref{estim-x}, item 3, $\left\Vert \dot{x}\right\Vert \in L^{2}\left(\left[0,+\infty\right[\right).$ Moreover,
by Proposition \ref{estim-x}, item  5, $\left\Vert v\right\Vert \in L^{\infty}\left(\left[0,+\infty\right[\right)$.
Hence, by (\ref{conv-Phi-4}), we obtain  $\frac{d}{dt}F_{2}\in L^{2}([0,+\infty[)$.
By Lemma \ref{lm:aux}, we deduce that $\lim_{t\rightarrow+\infty}F_{2}(t)=0$,
which is our claim.\\
 Note that the same conclusion can be obtained, by using the relation
\[
\Phi\left(x\left(t\right)\right)-\Phi\left(z\right)+\left\langle B\left(z\right),x\left(t\right)-z\right\rangle =-[\Phi(z)-\Phi(x(t))-\left\langle v(t),z-x(t)\right\rangle ]+\left\langle v(t)+B(z),x(t)-z\right\rangle ,
\]
 Proposition \ref{conv-Phi1}, and Theorem \ref{conv1}, item 3. \end{proof}

\begin{corollary} \label{strong-conv} Let us suppose that $\Phi$ is strongly
convex. Then the solution set $S$ is reduced to a single element
$\bar{z}$, and any orbit $x(\cdot)$ of system {\rm(\ref{2})-(\ref{3})} converges
strongly to $\bar{z}$, as $t\rightarrow+\infty$. 
\end{corollary}
 \begin{proof}
Since $\Phi$ is strongly convex, its subdifferential $\partial\Phi$
is strongly monotone, and so is the sum $A=\partial\Phi+B$. Thus
the solution set is reduced to a single element, let $\bar{z}$.\\
 Moreover, since $\Phi$ is strongly convex, and $-B(\bar{z})\in\partial\Phi(\bar{z})$,
we have the subdifferential inequality 
\[
\Phi\left(x\left(t\right)\right)-\Phi\left(\bar{z}\right)+\left\langle B\left(z\right),x\left(t\right)-\bar{z}\right\rangle \geq\gamma{\left\Vert x(t)-\bar{z}\right\Vert }^{2}
\]
 for some positive constant $\gamma$. By Proposition \ref{conv-Phi2}
\[
\Phi(x(t))-\Phi(\bar{z})+\left\langle B\bar{z},x(t)-\bar{z}\right\rangle \longrightarrow0\quad\mbox{as}\ t\rightarrow +\infty.
\]
 Hence $\lim\left\Vert x(t)-\bar{z}\right\Vert =0$, which gives the
claim. 
\end{proof}

\begin{corollary} \label{inf-compact} Suppose that $S\neq\emptyset$. Let us suppose that $\Phi$ is boundedly inf-compact, 
i.e., the intersections of the sublevel sets of $\Phi$ with closed balls of $\mathcal H$ 
are relatively compact sets. Then any orbit $x(\cdot)$ of system {\rm(\ref{2})-(\ref{3})} converges
strongly  as $t\rightarrow+\infty$, and its limit belongs to  $S$. 
\end{corollary}
 \begin{proof}
 We know that the orbit $x(\cdot)$ of system (\ref{2})-(\ref{3}) converges weakly to some $\bar{z} \in S$. As a consequence it is bounded. Moreover by
Proposition \ref{conv-Phi2}, 
 $\Phi(x(t)) \rightarrow \Phi(x_{\infty})$, and hence $x(\cdot)$ remains in a fixed sublevel set of $\Phi$. 
 Since $\Phi$ is boundedly inf-compact, we obtain that the trajectory is relatively compact 
 in $\mathcal H$, and thus converges strongly.
\end{proof}

\begin{remark} { \rm By a similar argument, strong convergence of  $x(\cdot)$ holds under the  Kadek-Klee property:
whenever $(x_k)$ converges weakly to $x$, and $(\Phi(x_k))$ converges to $\Phi(x)$, then $(x_k)$ converges strongly to $x$.}  
\end{remark}

\section{Forward-backward algorithms associated with the  regularized Newton dynamic}\label{algo}

Our algorithm is constructed using the following ideas. We rely on
the equivalent formulation of the dynamic involving the new variable
$y$ 
\begin{align}
 & x\left(t\right)=\mbox{prox}_{\mu\Phi}(y\left(t\right))\label{basic 30}\\
 & \dot{y}\left(t\right)+y\left(t\right)-\mbox{prox}_{\mu\Phi}(y(t))+\mu B\left(\mbox{prox}_{\mu\Phi}(y(t))\right)=0.\label{basic 31}
\end{align}
 The differential equation (\ref{basic 31}) is governed by a Lipschitz
continuous operator. Explicit discretization of (\ref{basic 31})
with respect to the time variable $t$, with constant step size $h>0$,
gives 
\[
\begin{cases}
x_{k}=\mbox{prox}_{\mu\Phi}(y_{k}),\\
\frac{y_{k+1}-y_{k}}{h}+y_{k}-\mbox{prox}_{\mu\Phi}(y_{k})+\mu B\left(\mbox{prox}_{\mu\Phi}(y_{k})\right)=0.
\end{cases}
\]
 The algorithm can be equivalently written as \ $(x_{k},y_{k})\rightarrow(x_{k},y_{k+1})\rightarrow(x_{k+1},y_{k+1})$,
\begin{equation}
{\rm (FBN)}\ \begin{cases}
x_{k}=\mbox{prox}_{\mu\Phi}(y_{k}),\\
\smallskip y_{k+1}=(1-h)y_{k}+h\left(x_{k}-\mu B\left(x_{k}\right)\right).
\end{cases}
\end{equation}
 When $h=1$ we recover the classical forward-backward algorithm 
\[
x_{k+1}=\mbox{prox}_{\mu\Phi}\Big(x_{k}-\mu B\left(x_{k}\right)\Big).
\]
(FBN) is closely related to the relaxed forward-backward
algorithm (\cite[Theorem 25.8]{BC}). It involves the same basic blocks
but in a different order. When the prox is linear, then the operations
commute, and we recover the classical relaxed forward-backward algorithm.
But, in general, for nonlinear problems, and in the case $\lambda$
non constant (which is of interest with respect to Newton method)
this is not the case. 
 As a guide for our study of the convergence of this algorithm, we
 use the Lyapunov functions that have been put to the fore in
the study of the continuous dynamics.

\noindent As a standing assumption we make the following set of hypotheses:

\smallskip{}

\noindent \textbf{Hypothesis $\mathbf{H}$:} 
\begin{itemize}
\item $ $$\mathbf{H}_{\Phi}$ : The function $\Phi: \mathcal H\to\mathbb{R}\cup \left\{+\infty\right\}$
is proper, lower-semicontinuous and convex. \smallskip{}

\item $\mathbf{H}_{B}$: The operator $B:\mathcal H\rightarrow \mathcal H$ is $\beta$-cocoercive.
\smallskip{}

\item $\mathbf{H}_{h,\mu}$: Parameters $h$ and $\mu$ satisfy $0<h<\delta:=\frac{1}{2}+\inf\left\{ 1,\frac{\beta}{\mu}\right\} $,
and $0<\mu<2\beta$. \smallskip{}

\item $\mathbf{H}_{S}$: The solution set $S=\left\{ z\in \mathcal H;\;\partial\Phi(z)+Bz\ni0\right\} $
is nonempty. 
\end{itemize}
Let us state our main convergence result.

\begin{theorem}\label{T:main} Let Hypothesis $\mathbf{H}$ hold. Let
$(x_{k},y_{k})$ be a sequence generated by {\em (FBN)}. Then the
following properties hold:

\smallskip{}

a) $(y_{k})$ converges weakly  to an element $\bar{y}$,
with $\mbox{{ {\rm prox}}}_{\mu\Phi}\bar{y}\in S$;

\smallskip{}

b) \ $(x_{k})$ converges weakly  to $\bar{x}=\mbox{{ {\rm prox}}}_{\mu\Phi}\bar{y}$,
with $\bar{x}\in S$;

\smallskip{}

c) \ $B(x_{k})$ converges strongly to $B\bar{x}$;

\smallskip{}

d) the velocity is square summable, i.e., $\sum_{k}\|x_{k}-x_{k-1}\|^{2}<\infty$,
and $\sum_{k}\|y_{k}-y_{k-1}\|^{2}<\infty$; in particular $x_{k}-x_{k-1}$
and $y_{k}-y_{k-1}$ converge strongly to zero.

\smallskip{}

e) $y_{k}-x_{k}$ converges strongly in $\mathcal H$. 
\end{theorem}

Before proving Theorem \ref{T:main}, we  review some classical results on $\alpha$-averaged operators, which will
be useful. 

\subsection {$\alpha$-averaged operators}\label{average}

We will use the notion of $\alpha$-averaged operator, see \cite[Definition 4.23]{BC}.
An operator $T:\mathcal H\to \mathcal H$ is $\alpha$-averaged with constant $0<\alpha<1$,
if there exists a nonexpansive operator $R:\mathcal H\to \mathcal H$ such that $T=(1-\alpha)I+\alpha R$.
The notions of cocoerciveness and $\alpha$-averaged are intimately
related. We collect below some classical facts that will be useful. 

\begin{itemize}
\item $T:\mathcal H\to \mathcal H$ is $\beta$-cocoercive iff $\beta T$ is $\frac{1}{2}$-averaged,
see \cite[Remark 4.24]{BC}. 
\item For any $\mu>0$, the operator $\mbox{prox}_{\mu\Phi}$ is $\frac{1}{2}$-averaged
(\cite[Corollary 23.8]{BC}). 
\item Suppose that $T:\mathcal H\to \mathcal H$ is $\beta$-cocoercive, and $0<\mu<2\beta$.
Then, the operator $(I-\mu B)$ is $\frac{\mu}{2\beta}$-averaged
(\cite[Proposition 4.33]{BC}). This result makes precise Lemma \ref{l:9}. 
\end{itemize}

A major interest of the notion of $\alpha$-averaged operator is that
the composition of two such operators is still an averaged operator.
This will be particularly useful when considering Krasnoselki-Mann
iteration for the mapping $T=(I-\mu B)\circ(\mbox{prox}_{\mu\Phi})$.
More precisely 

\begin{lemma}{\rm \cite[Proposition~4.32]{BC}}\label{l:90}
Let $T_{i}: \mathcal H\to \mathcal H$ be a $\alpha_{i}$-averaged operator, $i=1,2$.
Then $T:=T_{1}\circ T_{2}$ is $\alpha$-averaged with constant $\alpha=\frac{1}{\delta}$,
and $\delta=\frac{1}{2}+\frac{1}{2}\inf\left\{ \frac{1}{\alpha_{1}},\frac{1}{\alpha_{2}}\right\} $.
\end{lemma}

\subsection{Convergence of the (FBN) algorithm}

\begin{proof} We will first study the convergence of the sequence
$(y_{k})$, and then of the sequence $(x_{k})$.

\smallskip{}

\noindent \textbf{Convergence of the sequence $(y_{k})$}. \
It will be obtained as a direct consequence of the convergence of
the Krasnosel'ski-Mann iteration for nonexpansive mappings. Let 
 $T:\mathcal H\to \mathcal H$ be the operator which is defined by: for any $\xi \in \mathcal H$, 
\begin{equation}
T(\xi)=(I-\mu B)\circ(\mbox{prox}_{\mu\Phi}(\xi)).
\end{equation}
 The algorithm (FBN) can be equivalently written as 
\begin{equation}
\begin{cases}
x_{k}=\mbox{prox}_{\mu\Phi}(y_{k}),\\
\smallskip y_{k+1}=(1-h)y_{k}+hT(y_{k})\label{y-def}.
\end{cases}
\end{equation}
 i) Let us first suppose that $0<h<1$. Let us verify that we are
in the situation covered by Krasnosel'ski-Mann theorem. Since $B$
is $\beta$-cocoercive, and $0<\mu<2\beta$, by Lemma \ref{l:9},
the operator $Id-\mu B$ is nonexpansive. Since $\mbox{prox}_{\mu\Phi}$
is nonexpansive, we deduce that the composition mapping $T=(I-\mu B)\circ(\mbox{prox}_{\mu\Phi})$
is nonexpansive. By Krasnosel'ski-Mann algorithm, \cite[Theorem 5.14]{BC},
for $0<h<1$, the sequence $(y_{k})$ converges weakly
to a fixed point of $T$, let $y_{k}\to\bar{y}$ with $T(\bar{y})=\bar{y}$.
Equivalently, by definition of $T$, we have $\mbox{prox}_{\mu\Phi}(\bar{y})-\mu B(\mbox{prox}_{\mu\Phi}(\bar{y}))=\bar{y}$.
Set $\bar{z}:=\mbox{prox}_{\mu\Phi}(\bar{y})$. We have $\bar{z}-\mu B(\bar{z})=\bar{y}$,
which gives 
\[
\frac{1}{\mu}(\bar{y}-\mbox{prox}_{\mu\Phi}(\bar{y}))+B(\bar{z})=0.
\]
 The extremality condition characterizing $\mbox{prox}_{\mu\Phi}(\bar{y})$
gives 
\[
\frac{1}{\mu}(\bar{y}-\mbox{prox}_{\mu\Phi}(\bar{y}))\in\partial\Phi(\mbox{prox}_{\mu\Phi}(\bar{y})).
\]
 Comparing the two above equations, we finally obtain 
\[
\partial\Phi(\bar{z})+B(\bar{z})\ni0.
\]

\smallskip{}

ii) Take now $h$ possibly greater or equal than $1$, but $h<\delta$.
Let's analyze in more detail the algorithm 
\[
y_{k+1}=(1-h)y_{k}+hT(y_{k}).
\]
 We rely on the notion of $\alpha$-averaged operator that has been
discussed in the previous subsection. Let us examine the operator $T=(I-\mu B)\circ(\mbox{prox}_{\mu\Phi})$.
The operator $\mbox{prox}_{\mu\Phi}$ is $\frac{1}{2}$-averaged (\cite[Corollary 23.8]{BC}).
The operator $(I-\mu B)$ is $\frac{\mu}{2\beta}$-averaged (\cite[Proposition 4.33]{BC}).
Hence $T$ is $\alpha$-averaged with constant $\alpha=\frac{1}{\delta}$,
and $\delta=\frac{1}{2}+\inf\left\{ 1,\frac{\beta}{\mu}\right\} $,
see Lemma \ref{l:90}. By the condition $0<\mu<2\beta$, we have $\delta>1$,
and hence $0<\alpha<1$. By \cite[Proposition 5.15]{BC}, we deduce
that, for any $0<h<\delta$ the sequence $(y_{k})$ converges weakly
to a fixed point of $T$, and 
\[
\sum_{k}\|Ty_{k}-y_{k}\|^{2}<\infty.
\]
 Equivalently, by (\ref{y-def}) 
\begin{equation}
\sum_{k}\|y_{k+1}-y_{k}\|^{2}<\infty.\label{y-est}
\end{equation}
 Hence 
\begin{equation}
y_{k+1}-y_{k}\to0\ \mbox{strongly in } \mathcal H.\label{y-est-2}
\end{equation}

\smallskip{}

\noindent \textbf{A Lyapunov-type sequence} Take $\bar{z}$ an arbitrary
element in $S$. Since $x_{k}=\mbox{prox}_{\mu\Phi}(y_{k})$, we have
\begin{equation}
v_{k}:=\frac{1}{\mu}(y_{k}-x_{k})\in\partial\Phi(x_{k}).\label{E:Liap51}
\end{equation}
 As a Lyapunov sequence take 
\begin{equation}
A_{k}:=\frac{1}{2\mu}\|x_{k}-\bar{z}\|^{2}+g_{\bar{z}}^{k},\label{E:Liap8}
\end{equation}
 with 
\begin{equation}
g_{\bar{z}}^{k}:=\Phi(\bar{z})-\left(\Phi(x_{k})+\langle\bar{z}-x_{k},v_{k}\rangle\right).\label{E:Liap5}
\end{equation}
 a) The following equality is a direct consequence of the Hilbert
structure of $\mathcal H$. 
\begin{align}
\|x_{k+1}-\bar{z}\|^{2} & =\|x_{k+1}-x_{k}+x_{k}-\bar{z}\|^{2}\\
 & =\|x_{k+1}-x_{k}\|^{2}+2\langle x_{k+1}-x_{k},x_{k}-\bar{z}\rangle+\|x_{k}-\bar{z}\|^{2}.\label{E:Liap4}
\end{align}
 b) We have 
\begin{align*}
g_{\bar{z}}^{k+1}-g_{\bar{z}}^{k} & =\Phi(\bar{z})-\Phi(x_{k+1})-\langle\bar{z}-x_{k+1},v_{k+1}\rangle-\Phi(\bar{z})+\Phi(x_{k})+\langle\bar{z}-x_{k},v_{k}\rangle\\
 & =\Phi(x_{k})-\Phi(x_{k+1})+\langle x_{k+1}-x_{k},v_{k}\rangle+\langle x_{k+1}-\bar{z},v_{k+1}-v_{k}\rangle.
\end{align*}
 By convexity of $\Phi$ and $v_{k}\in\partial\Phi(x_{k})$, we have
$\Phi(x_{k})-\Phi(x_{k+1})+\langle x_{k+1}-x_{k},v_{k}\rangle\leq0$,
which gives 
\begin{equation}
g_{\bar{z}}^{k+1}-g_{\bar{z}}^{k}\leq\langle x_{k+1}-\bar{z},v_{k+1}-v_{k}\rangle.\label{E:Liap7}
\end{equation}
 c) Let us show that $A_{k}:=\frac{1}{2\mu}\|x_{k}-\bar{z}\|^{2}+g_{\bar{z}}^{k}$
is a Lyapunov sequence. By \eqref{E:Liap51}, \eqref{E:Liap4}, and
\eqref{E:Liap7} we have 
\begin{align*}
A_{k+1}-A_{k} & \leq\frac{1}{2\mu}\left(\|x_{k+1}-x_{k}\|^{2}+2\langle x_{k+1}-x_{k},x_{k}-\bar{z}\rangle\right)+\langle x_{k+1}-\bar{z},v_{k+1}-v_{k}\rangle\\
 & \leq\frac{1}{2\mu}\|x_{k+1}-x_{k}\|^{2}+\frac{1}{\mu}\left(\langle x_{k+1}-x_{k},x_{k}-\bar{z}\rangle+\langle x_{k+1}-\bar{z},(y_{k+1}-y_{k})-(x_{k+1}-x_{k})\rangle\right)\\
 & \leq\frac{1}{2\mu}\|x_{k+1}-x_{k}\|^{2}+\frac{1}{\mu}\left(-\|x_{k+1}-x_{k}\|^{2}+\langle x_{k+1}-\bar{z},y_{k+1}-y_{k}\rangle\right)\\
 & \leq-\frac{1}{\mu}\|x_{k+1}-x_{k}\|^{2}+\frac{1}{\mu}\langle x_{k+1}-\bar{z},y_{k+1}-y_{k}\rangle\\
 & \leq-\frac{1}{\mu}\|x_{k+1}-x_{k}\|^{2}+\frac{1}{\mu}\langle x_{k}-\bar{z},y_{k+1}-y_{k}\rangle+\frac{1}{\mu}\langle x_{k+1}-x_{k},y_{k+1}-y_{k}\rangle.
\end{align*}
 Let us write (FBN) algorithm as 
\begin{equation}
y_{k+1}-y_{k}=h[(x_{k}-y_{k})-\mu Bx_{k}].\label{E:Liap11}
\end{equation}
 Replacing $y_{k+1}-y_{k}$ by this expression in the above inequality
gives 
\[
A_{k+1}-A_{k}\leq-\frac{1}{\mu}\|x_{k+1}-x_{k}\|^{2}+\frac{h}{\mu}\langle x_{k}-\bar{z},(x_{k}-y_{k})-\mu Bx_{k}\rangle+\frac{1}{\mu}\langle x_{k+1}-x_{k},y_{k+1}-y_{k}\rangle.
\]
 Equivalently 
\begin{equation}
A_{k+1}-A_{k}+\frac{1}{\mu}\|x_{k+1}-x_{k}\|^{2}+h\langle Bx_{k}-B\bar{z},x_{k}-\bar{z}\rangle+h\langle x_{k}-\bar{z},\frac{1}{\mu}(y_{k}-x_{k})+B\bar{z}\rangle\leq\frac{1}{\mu}\langle x_{k+1}-x_{k},y_{k+1}-y_{k}\rangle.\label{E:Liap25}
\end{equation}
 Since $B$ is $\beta$-cocoercive 
\begin{equation}
\langle Bx_{k}-B\bar{z},x_{k}-\bar{z}\rangle\geq\beta\|Bx_{k}-B\bar{z}\|^{2}.\label{E:Liap26}
\end{equation}
 By (\ref{E:Liap51}), we have $\frac{1}{\mu}(y_{k}-x_{k})\in\partial\Phi(x_{k})$.
By definition of $S$, and $\bar{z}\in S$, we have $-B\bar{z}\in\partial\Phi(\bar{z})$.
Hence, by monotonicity of $\partial\Phi$ 
\begin{equation}
\langle x_{k}-\bar{z},\frac{1}{\mu}(y_{k}-x_{k})+B\bar{z}\rangle\geq0.\label{E:Liap27}
\end{equation}
 Combining (\ref{E:Liap25}), (\ref{E:Liap26}), and (\ref{E:Liap27}),
we obtain 
\begin{equation}
A_{k+1}-A_{k}+\frac{1}{\mu}\|x_{k+1}-x_{k}\|^{2}+h\beta\|Bx_{k}-B\bar{z}\|^{2}\leq\frac{1}{\mu}\langle x_{k+1}-x_{k},y_{k+1}-y_{k}\rangle.\label{E:Liap28}
\end{equation}
 By using $\langle x_{k+1}-x_{k},y_{k+1}-y_{k}\rangle\leq\frac{1}{2}\|x_{k+1}-x_{k}\|^{2}+\frac{1}{2}\|y_{k+1}-y_{k}\|^{2}$,
we obtain 
\begin{equation}
A_{k+1}-A_{k}+\frac{1}{2\mu}\|x_{k+1}-x_{k}\|^{2}+h\beta\|Bx_{k}-B\bar{z}\|^{2}\leq\frac{1}{2\mu}\|y_{k+1}-y_{k}\|^{2}.\label{E:Liap29}
\end{equation}
 By (\ref{y-est}), $\sum_{k}\|y_{k+1}-y_{k}\|^{2}<\infty$, and $A_{k}$
is nonnegative. By a standard argument, from (\ref{E:Liap29}), we obtain 

\smallskip{}

\noindent a) $\lim A_{k}$ exists. Since $g^{k}$ is nonnegative,
we have $A_{k}\geq\frac{1}{2\mu}\|x_{k}-\bar{z}\|^{2}$. As a consequence,
the sequence $(x_{k})$ is bounded.

\smallskip{}

\noindent b) $\sum_{k}\|Bx_{k}-B\bar{z}\|^{2}<\infty$. Hence 
\begin{equation}
B(x_{k})\to B\bar{z}\  \ \mbox{ strongly}\label{E:Liap39}
\end{equation}
 where $B\bar{z}$ is uniquely defined for $\bar{z}\in S$.

\smallskip{}

\noindent c) $\sum_{k}\|x_{k+1}-x_{k}\|^{2}<\infty$.

\smallskip{}

\noindent This proves item c) and d) of Theorem \ref{T:main}.

\smallskip{}

\noindent \textbf{Convergence of the sequence $(x_{k})$}. Let us
write (FBN) in the following form 
\begin{equation}
\frac{1}{h\mu}(y_{k+1}-y_{k})+\frac{1}{\mu}(y_{k}-x_{k})+B(x_{k})=0,\label{E:Liap40}
\end{equation}
 with 
\begin{equation}
v_{k}:=\frac{1}{\mu}(y_{k}-x_{k})\in\partial\Phi(x_{k}).\label{E:Liap41}
\end{equation}
 By (\ref{y-est-2}), $y_{k+1}-y_{k}\to0\ \mbox{strongly in } \mathcal H$.
By (\ref{E:Liap39}), $B(x_{k})\to B\bar{z}$ strongly in $\mathcal H$. From
(\ref{E:Liap40}) we deduce that 
\[
y_{k}-x_{k}\to-\mu B\bar{z}
\]
 strongly in $\mathcal H$. Note again that $B\bar{z}$ is uniquely defined
when $\bar{z}\in S$. Since we have already obtained that the sequence
$(y_{k})$ converges weakly, we deduce that the sequence $(x_{k})$
converges weakly, let $x_{k}\rightharpoonup\bar{x}$ weakly. From
(\ref{E:Liap40}) and (\ref{E:Liap41}) we have 
\[
-\frac{1}{h\mu}(y_{k+1}-y_{k})\in(\partial\Phi+B)(x_{k}).
\]
 The operator $A=\partial\Phi+B$ is maximal monotone, and hence is
demi-closed. From $y_{k+1}-y_{k}\rightarrow0$ strongly (\ref{y-est-2}),
$x_{k}\rightharpoonup\bar{x}$ weakly, we deduce that $A(\bar{x})=\partial\Phi(\bar{x})+B(\bar{x})\ni0$,
that is $\bar{x}\in S$.

Let us make precise the relation between the respective limits of
the sequences $(y_{k})$ and $(x_{k})$. Let $y_{k}\rightharpoonup\bar{y}$
weakly. Since $\bar{x}\in S$ we have $B(x_{k})\to B\bar{x}$ strongly
in $\mathcal H$. From (\ref{E:Liap40}), we deduce that $\bar{y}-\bar{x}+\mu B\bar{x}=0$.
Since $B\bar{x}+\partial\Phi(\bar{x})\ni0$, we obtain $\bar{x}+\mu\partial\Phi(\bar{x})\ni\bar{y}$.
Hence $\bar{x}=\mbox{{ {\rm prox}}}_{\mu\Phi}\bar{y}$.

This complete the proof of Theorem \ref{T:main}.
\end{proof}

\begin{remark} { \rm a) Clearly, since $\delta>1$, we can take an arbitrary
$0<h\leq1$. Indeed, the above analysis provides an over-relaxation
result.\\
 b) The above result can be readily extended to the case $h_{k}$
varying with $k$. The convergence of $(y_{k})$ is satisfied under the assumption: there exists some $\epsilon>0$,
such that for all $k\in\mathbb{N}$, \ $0<\epsilon\leq h_{k}\leq\delta-\epsilon$.\\
 c) When the dimension of $\mathcal H$ is finite, by continuity of $\mbox{prox}_{\mu\Phi}$,
and $x_{k}=\mbox{prox}_{\mu\Phi}(y_{k})$ we immediately obtain the
convergence of the sequence $(x_{k})$ to an element of the solution
set $S$. In the infinite dimensional case, this argument does not
work anymore, because of the lack of continuity of the prox mapping
for the weak topology.\\
 d) By analogy with the continuous case, one can reasonably conjecture
that the weak convergence of the sequence $(x_{k})$ holds under the
weaker condition: $0<h<1$ \ and \ $h\mu<2\beta$.}
  \end{remark}

\begin{remark} { \rm Comparing  the numerical performance of the forward-backward algorithms provided by discretization of various dynamical systems is an important issue. 
This is a delicate question, directly related to obtaining rapid numerical methods, a subject of ongoing study,  see \cite {APR2}, \cite{APR}.}
  \end{remark}

\section{Semigroup generated by $-(\partial\Phi+B)$, and FB algorithms}\label{sgroup}

\subsection{Continuous case}

Consider a closely related dynamical system, which is the semigroup
generated by $-A$; $A=\partial\Phi+B$, whose orbits
are the solution trajectories of the differential inclusion 
\begin{equation}
\dot{x}\left(t\right)+\partial\Phi(x(t))+B\left(x\left(t\right)\right)\ni0.\label{sg1}
\end{equation}
Since the operator $A=\partial\Phi+B$ is maximal
monotone, (\ref{sg1}) is relevant to the general theory of semigroups
generated by maximal monotone operators. For any Cauchy data $x_{0}\in{\mbox{dom}\partial\Phi}$,
there exists a unique strong solution of (\ref{sg1}) which satisfies
$x(0)=x_{0}$, see \cite{Br}. Moreover, by a direct adaptation of the results of \cite[Theorem 3.6]{Br}, one can verify that there is a regularizing effect on the initial condition: for 
 $x_{0}\in\overline{\mbox{dom}\Phi}$, there exists a unique strong solution of (\ref{sg1}) with Cauchy data $x(0)= x_0$, and which satisfies $x(t) \in \mbox{dom}\partial \Phi$ for all $t>0$.

 Let us suppose that $S\neq\emptyset$, where $S$
still denotes the set of zeroes of $A=\partial\Phi+B$. 
 Following Baillon-Br\'ezis \cite{BB}, each orbit of (\ref{sg1}) converges
weakly, in an ergodic way, to an equilibrium, which is an element of $S$.
Note that the convergence theory of Bruck does apply separately to
$\partial\Phi$ and $B$, which are demipositive, see \cite{Bruck}.
But it is not known if the sum of the two operators $\partial\Phi+B$
is still demipositive. Indeed, it is not clear whether this notion
is stable by sum. Thus, we are naturally led to perform a direct study
of the convergence properties of the orbits of (\ref{sg1}). Surprisingly,
we have not found references to a previous systematic study of this
question. Indeed, we are going to show that (\ref{sg1}) has convergence
properties which are similar to the regularized Newton-like dynamic.
Then, we shall compare and show the differences between the two systems.

\begin{theorem}\label{sg2} Suppose that $S\neq\emptyset$. Then, for
any orbit $x(\cdot)$ of {\rm(\ref{sg1})}, the following properties hold:

\smallskip{}

1. $\int_{0}^{+\infty}\left\Vert \dot{x}\left(t\right)\right\Vert ^{2}dt<+\infty$,
i.e., $x(\cdot)$ has a finite energy.

\vspace{1mm}

2. \ $x(\cdot)$ converges weakly to an element of $S$.

\vspace{1mm}

3. \ $B(x(\cdot))$ converges strongly to $Bz$, where $Bz$ is uniquely
defined for $z\in S$. 
\end{theorem}
\begin{proof} Let $x\left(\cdot\right): [0,+\infty [  \rightarrow \mathcal H$
be an orbit of (\ref{sg1}). Equivalently, we set 
\begin{align}
 & v(t)\in\partial\Phi(x(t))\label{sg30}\\
 & \dot{x}\left(t\right)+v(t)+B\left(x\left(t\right)\right)=0.\label{sg31}
\end{align}
 For any $z\in S$, let us define $h_{z}: [0,+\infty [\rightarrow{\mathbb{R}}^{+}$
by 
\[
h_{z}\left(t\right):=\frac{1}{2}\left\Vert x\left(t\right)-z\right\Vert ^{2}.
\]
 Let us show that $h_{z}$ is a Lyapunov function. By the classical
derivation chain rule, and (\ref{sg31}), for almost all $t\geq0$
\begin{align}
\frac{d}{dt}h_{z}\left(t\right) & =\left\langle x\left(t\right)-z,\dot{x}\left(t\right)\right\rangle \label{sg4}\\
 & =-\left\langle x\left(t\right)-z,v(t)+B\left(x\left(t\right)\right)\right\rangle .
\end{align}
 Let us rewrite this last equality as 
\begin{equation}
\frac{d}{dt}h_{z}\left(t\right)+\left\langle x\left(t\right)-z,v(t)+Bz\right\rangle +\left\langle x\left(t\right)-z,B\left(x\left(t\right)\right)-Bz\right\rangle =0.\label{sg5}
\end{equation}
 Since $z\in S$, we have $-Bz\in\partial\Phi(z)$. Moreover, $v(t)\in\partial\Phi(x(t))$.
By monotonicity of $\partial\Phi$, this gives 
\begin{equation}
\left\langle x\left(t\right)-z,v(t)+Bz\right\rangle \geq0.\label{sg6}
\end{equation}
 Combining (\ref{sg5}) with (\ref{sg6}) we obtain 
\begin{equation}
\frac{d}{dt}h_{z}\left(t\right)+\left\langle x\left(t\right)-z,B\left(x\left(t\right)\right)-Bz\right\rangle \leq0.\label{sg7}
\end{equation}
 By cocoercivity of $B$, we deduce that 
\begin{equation}
\frac{d}{dt}h_{z}\left(t\right)+\beta\left\Vert B\left(x\left(t\right)\right)-Bz\right\Vert ^{2}\leq0.\label{sg8}
\end{equation}
 From this, we readily obtain that 
\begin{equation}
t\mapsto h_{z}\left(t\right)\ \mbox{is a decreasing function},\label{sg9}
\end{equation}
 and after integration of (\ref{sg8}) 
\begin{equation}
\int_{0}^{+\infty}\left\Vert B\left(x\left(t\right)\right)-Bz\right\Vert ^{2}dt<+\infty.\label{sg10}
\end{equation}
 Let us return to (\ref{sg8}). By (\ref{sg31}), we equivalently
have 
\begin{equation}
\frac{d}{dt}h_{z}\left(t\right)+\beta\left\Vert \dot{x}\left(t\right)+v(t)+Bz\right\Vert ^{2}\leq0.\label{sg11}
\end{equation}
 After developing 
\begin{equation}
\frac{d}{dt}h_{z}\left(t\right)+\beta\left\Vert \dot{x}(t)\right\Vert ^{2}+\beta\left\Vert v(t)+Bz\right\Vert ^{2}+2\beta\left\langle \dot{x}(t),v(t)+Bz\right\rangle \leq0.\label{sg12}
\end{equation}
 By using the derivation chain rule for a convex lower semicontinuous
function, see Lemma \ref{deriv-chain} 
\begin{equation}
\frac{d}{dt}\Phi\left(x\left(t\right)\right)=\left\langle \upsilon\left(t\right),\dot{x}\left(t\right)\right\rangle,\label{sg13}
\end{equation}
 we can rewrite (\ref{sg12}) as 
\begin{equation}
\frac{d}{dt}\left[h_{z}\left(t\right)+2\beta\Phi\left(x(t)\right)+2\beta\left\langle x(t),Bz\right\rangle \right]+\beta\left\Vert \dot{x}(t)\right\Vert ^{2}+\beta\left\Vert v(t)+Bz\right\Vert ^{2}\leq0.\label{sg14}
\end{equation}
 Set 
\begin{equation}
k_{z}\left(t\right):=\Phi\left(x(t)\right)-\Phi(z)+\left\langle Bz,x(t)-z\right\rangle .\label{sg15}
\end{equation}
 Since $-Bz\in\partial\Phi(z)$, by the convex subdifferential inequality,
we have $k_{z}\left(t\right)\geq0$. Let us rewrite (\ref{sg14})
as 
\begin{equation}
\frac{d}{dt}\left[h_{z}\left(t\right)+2\beta k_{z}\left(t\right)\right]+\beta\left\Vert \dot{x}(t)\right\Vert ^{2}+\beta\left\Vert v(t)+Bz\right\Vert ^{2}\leq0.\label{sg16}
\end{equation}
 Since $h_{z}+2\beta k_{z}$ is nonnegative, by integration of (\ref{sg16})
we infer 
\begin{equation}
\int_{0}^{+\infty}\left\Vert \dot{x}\left(t\right)\right\Vert ^{2}dt<+\infty.\label{sg17}
\end{equation}
 That's item 1. In order to prove item 2., which is the weak convergence
property of $x$, we use Opial's lemma \ref{Opial}, with $S$ equal to the solution set of problem (\ref{basicpb}).
 By (\ref{sg9}), $t\mapsto h_{z}\left(t\right)$
is a decreasing function, and hence $\lim\left\Vert x\left(t\right)-z\right\Vert $
exists. Let us complete the verification of the hypothesis of Opial's
lemma, by showing that every weak sequential cluster point of $x$
belongs to $S$. Let $\bar{x}=w-\lim x(t_{n})$ for some sequence
$t_{n}\rightarrow+\infty$. By the general theory of semigroups generated
by maximal monotone operators, we have that 
\begin{equation}
t\mapsto\left\Vert (\partial\Phi(x(t))+B(x(t)))^{0}\right\Vert \label{sg18}
\end{equation}
 is a nonincreasing function, where $(\partial\Phi(x(t))+B(x(t)))^{0}$
is the element of minimal norm of the closed convex set $\partial\Phi(x(t))+B(x(t))$.
Since $\dot{x}(t)=-(\partial\Phi(x(t))+B(x(t)))^{0}$ for almost all
$t\geq0$, we deduce from (\ref{sg17}) that 
\begin{equation}
\int_{0}^{+\infty}\left\Vert (\partial\Phi(x(t))+B(x(t)))^{0}\right\Vert ^{2}dt<+\infty.\label{sg19}
\end{equation}
 Since $t\mapsto\left\Vert (\partial\Phi(x(t))+B(x(t)))^{0}\right\Vert $
is nonincreasing, it converges and, by (\ref{sg19}) its limit is equal
to zero. Thus, by taking $w(t)=(\partial\Phi(x(t))+B(x(t)))^{0}$,
we have obtained the existence of a mapping $w$ which verifies: $w(t)\in(\partial\Phi+B)(x(t)$
for all $t>0$, and $w(t)\to 0$ strongly in $\mathcal H$, as $t\to +\infty$.
From $w(t_{n})\in(\partial\Phi+B)(x(t_{n})$, by the demiclosedness
property of the maximal monotone operator $A=\partial\Phi+B$, we
obtain $A(\bar{x})=\partial\Phi(\bar{x})+B(\bar{x})\ni0$, that is
$\bar{x}\in S$.

Let us now prove item 3. Set $F_{3}(t)=\left\Vert B(x(t))-Bz\right\Vert $.
By (\ref{sg10}) $F_{3}\in L^{2}\left(\left[0,+\infty\right[\right).$
Since $B$ is Lipschitz continuous and $\dot{x}\in L^{2}\left(\left[0,+\infty\right[\right)$
we obtain that $\frac{d}{dt}F_{3}\in L^{2}([0,+\infty[)$. By Lemma
\ref{lm:aux}, we deduce that $\lim_{t\rightarrow+\infty}F_{3}(t)=0$,
which is our claim. 
\end{proof}

\begin{remark} {\rm The strong convergence of orbits  falls within the general theory of
 semigroup of contractions generated by a maximal monotone operator $A$. It is satisfied if $A$ is strongly monotone, or $\Phi$ boundedly inf-compact (note that in the proof of Theorem
\ref{sg2} we have shown that $\Phi (x(t))$ converges, and thus is bounded).
Also note that, following \cite[Theorem 3.13]{Br}, if
$\textrm{int}S=\textrm{int}A^{-1}(0) \neq \emptyset$,  then (\ref{sg1}) has orbits whose total variation is bounded, and hence which converge strongly.}
\end{remark}

\begin{remark} {  \rm Let us compare the asymptotic behavior of the orbits
of the semigroup generated by $-(\partial\Phi+B)$ with the orbits of
the Newton-like regularized system. Since both converge weakly to
equilibria, the point is compare their rate of convergence. For simplicity
take $B=0$, and $\Phi$ convex differentiable. Thus the point is:
at which rate does $\nabla\Phi(x(t)$ converges to zero?

\noindent a) For the semigroup, the standard estimation is the linear
convergence: 
\[
\|\nabla\Phi(x(t)\|\leq\frac{C}{t}.
\]
 Indeed, without any further assumption on $\Phi$ or $\mathcal H$, this is
the best known general estimate. Indeed, in infinite dimensional spaces
one can exhibit orbits of the gradient flow which have infinite length,
this is a consequence of Baillon counterexample \cite{Ba}. Note that,
in finite dimensional spaces, the corresponding result is not known
\cite{Dani}.

\smallskip{}

\noindent b) For the Newton-like regularized system, $v(t)=\nabla\Phi(x(t)$
satisfies the differential equation 
\[
\lambda(t)\dot{x}\left(t\right)+\dot{\upsilon}\left(t\right)+\upsilon\left(t\right)=0.
\]
 By taking $\lambda(t)=ce^{-t}$, we have the following estimation
(see \cite[Proposition 5.1]{AS}) 
\[
\|\nabla\Phi(x(t)\|\leq Ce^{-t}.
\]
 These results naturally suggest to extend our results to the case
of a vanishing regularization parameter. }
\end{remark}
	
\subsection{Implicit/explicit time discretization: FB algorithm}

The discretization of (\ref{sg1})
with respect to the time variable $t$, in an implicit way with respect to the nonsmooth term $\partial \Phi$, and explicit with respect to the smooth term $B$, and with constant step size $h>0$,
gives 
\begin{equation}
\frac{x_{k+1}-x_{k}}{h}+\partial \Phi(x_{k+1})+ B(x_{k}) \ni 0.
\end{equation}
Equivalently
\begin{equation}
x_{k+1} = \left( I + h \partial \Phi\right)^{-1}\left( x_k  - h B(x_{k})\right).
\end{equation}
This is the classical forward-backward algorithm, whose convergence has been well established. 
The weak convergence of $(x_k)$  to an element of $S$ is obtained under the stepsize limitation: $0 < h < 2\beta$.
One can consult \cite{LM}, \cite[Theorem 25.8]{BC}, for the proof, and some further extensions of this result.

\section{Proximal-gradient dynamics and relaxed FB algorithms}\label{pg. section}
First recall some standard facts about the  continuous gradient-projection  system. This will lead us to consider a more general proximal-gradient dynamic. 
Then we will examine the corresponding relaxed FB algorithms, obtained by time discretization.

\subsection{Gradient-projection dynamics}
First take $\Phi = \delta_C$  equal to the indicator function of a closed convex set $C \subset \mathcal H$, and $B = \nabla \Psi$, 
the gradient of a convex differentiable function $\Psi: \mathcal H \to \mathbb R$.
The  semigroup of contractions, generated by $-A$, which has been studied in the previous section,  specializes in gradient-projection system
 \begin{equation}
\dot{x} (t) = \mbox{{\rm proj}}_{T_C (x(t))}\left(  - \nabla \Psi (x(t))  \right),
\end{equation} 
where $T_C (x)$ is the tangent cone to $C$ at $x\in C$. This is a direct consequence of the lazy property satisfied by the orbits of the  semigroup of contractions, generated by $-A$, see
\cite{Br}, and of the Moreau decomposition theorem in a Hilbert space (with respect to the tangent cone $T_C (x)$ and its polar cone $N_C (x)$).
From the perspective of optimization, this system has several  drawbacks. 
The orbits ignore the constraint until
they meet the boundary of $C$. Moreover, the vector field which governs the dynamic is discontinuous (at the boundary of the constraint).
The following system first considered by Antipin \cite{Ant}, and Bolte \cite{Bo} overcomes some of these difficulties:
\begin{equation}\label{gp00}
  \dot{x}(t) + x(t) - \mbox{{\rm proj}}_C  \left( x(t) - \mu \nabla \Psi (x(t))  \right) =0.
\end{equation}
It can be introduced in a natural way, by rewriting the optimality condition
 \begin{equation}\label{gp6}
\nabla \Psi (x) + N_C (x) \ni 0 
\end{equation}
as a fixed point problem
\begin{equation}\label{gp7}
 x - \mbox{{\rm proj}}_C \left( x - \mu \nabla \Psi (x)  \right)=0,
\end{equation}
where $\mu$ is a positive parameter (arbitrarily chosen). 
Note that the stationary points of (\ref{gp00}) are precisely the solutions of (\ref{gp7}).
This dynamic is  governed by a Lipschitz continuous vector field,  and the orbits are  classical solutions, i.e., 
continuously differentiable.
Its properties are 
 summarized in the following proposition, see \cite{Bo}.
 \begin{proposition} \label{gp-Thm2} 
Let  $\Psi:  \mathcal H \rightarrow \mathbb R$ be a convex differentiable function, whose gradient is Lipschitz continuous on bounded sets. Let $C$
  be a closed convex set in $\mathcal H$, and suppose that  $\Psi$ is bounded from below on $C$. 
Then, for any $x_0  \in \mathcal H$, there exists a unique classical global solution $x: [0, +\infty[ \rightarrow \mathcal H$  of the Cauchy problem for the relaxed gradient-projection dynamical system
\begin{equation}\label{gp9}
\left\{
\begin{array}{l}
 \dot{x}(t) + x(t) - \mbox{{\rm proj}}_C \left( x(t) - \mu \nabla \Psi (x(t))  \right)=0; \\
x(0)=x_0.
\end{array}
\right.
\end{equation}
The following asymptotic properties are satisfied:

\smallskip

i) If $S=  {\rm \mbox{argmin}}_C \Psi$ is nonempty,
then $x(t)$  converges weakly to some $x_{\infty} \in S$, as $t \rightarrow + \infty$.

\smallskip

ii) If moreover $x_0  \in C$, then $x(t) \in C$ for all $t \geq 0$,     $\Psi (x(t)$ decreases to $\inf _C \Psi$ as $t$ increases to $+ \infty$, and  
\begin{equation}\label{gp10}
  \mu \frac{d}{dt}\Psi(x(t)) + \| \dot{x}(t) \|^2  \leq 0.
\end{equation}
\end{proposition} 

\subsection{Proximal-gradient dynamics}
Let us now return to our setting: $A= \partial \Phi + B$, where $\Phi$ is a closed convex proper function, and 
$B$ is a monotone cocoercive operator (the previous case corresponds to  $\Phi = \delta_C$ and  $B = \nabla \Psi$, with $\nabla \Psi$ Lipschitz continuous).
As a natural extension of (\ref{gp00}), let us consider the differential system
\begin{equation}\label{pg00}
  \dot{x}(t) + x(t) - \mbox{{\rm prox}}_{\mu \Phi}  \left( x(t) - \mu B(x(t))  \right) =0.
\end{equation}
We shall see that the explicit discretization of this system gives the relaxed FB algorithm.
The vector field which governs (\ref{pg00}) is Lipschitz continuous. Hence, for any $x_0 \in \mathcal H$, the corresponding Cauchy problem has a unique global classical solution.
As far as we know, the convergence properties of this system have not been studied in this framework.
 Let us state our results.

\begin{theorem} \label{pg-Thm3} 
Let   $\Phi:\mathcal H\rightarrow\mathbb{R}\cup\left\{ +\infty\right\} $  be a convex lower semicontinuous proper function, and
$B$ a maximal monotone operator which is $\beta$-cocoercive.
Suppose that  $S=\left\{ z\in \mathcal H;\;\partial\Phi(z)+Bz\ni0\right\} $, the solution
set of {\rm(\ref{basicpb})}, is  nonempty.

For any $x_0  \in \mathcal H$,  let $x: [0, +\infty[ \rightarrow \mathcal H$  be the unique classical global solution of the Cauchy problem for the proximal-gradient dynamical system
\begin{equation}\label{pg9}
\left\{
\begin{array}{l}
 \dot{x}(t) + x(t) - \mbox{{\rm prox}}_{\mu \Phi}  \left( x(t) - \mu B(x(t))  \right)=0; \\
x(0)=x_0.
\end{array}
\right.
\end{equation}
 Then,  the following asymptotic properties are satisfied:

\smallskip

1. Suppose that \ $0< \mu < 4 \beta$, then
 
 i) $x(t)$  converges weakly to some $x_{\infty} \in S$, as $t \rightarrow + \infty$.

\smallskip

ii)  $B(x(t))$ converges strongly to $Bz$ as $t \rightarrow + \infty$,   where $Bz$ is uniquely defined for $z \in S$.

\smallskip

iii) $\lim_{t\to+\infty}\dot{x}(t)   = 0$ and $\int_0^{\infty} \| \dot{x}(t) \|^2  dt < + \infty$.

\smallskip

2. Suppose that $B= \nabla \Psi$, where $\Psi$ is a convex differentiable function. Then, for arbitrary $ \mu > 0$, the above properties i), ii), iii) are satisfied.

\end{theorem} 

\begin{proof} We rely on a Lyapunov analysis.
Take $z\in S$. Equivalently 
\begin{equation}\label{pg10}
-Bz \in \partial \Phi (z). 
\end{equation}
Set $\xi (t) :=  x(t) - \mu B(x(t))$. By definition of $\mbox{{\rm prox}}_{\mu \Phi}$, 
 we have 
 $$\frac{1}{\mu}(\xi (t) - \mbox{{\rm prox}}_{\mu \Phi} \xi (t)) \in \partial \Phi ( \mbox{{\rm prox}}_{\mu \Phi} \xi (t)).$$
Since $  \mbox{{\rm prox}}_{\mu \Phi} \xi (t) =  x(t) + \dot{x}(t) $, the above equation can be written in equivalent way
\begin{equation}\label{pg11}
 - B(x(t)) -  \frac{1}{\mu} \dot{x}(t) \in \partial \Phi ( x(t) + \dot{x}(t)).
\end{equation}
By the monotonicity property of the operator $\partial \Phi$, and (\ref{pg10}), (\ref{pg11}), we obtain
\begin{equation}\label{pg13}
0 \geq   \left\langle  x(t)  -z + \dot{x}(t),  B(x(t)) -Bz + \frac{1}{\mu}  \dot{x}(t)   \right\rangle.
\end{equation}
Equivalently
\begin{equation}\label{pg14}
0 \geq   \frac{1}{2\mu} \frac{d}{dt} \| x(t)  -z  \|^2  + \frac{1}{\mu} \| \dot{x}(t) \|^2   + \left\langle  B(x(t)) -Bz , x(t)  -z  \right\rangle + \left\langle \dot{x}(t),  B(x(t)) -Bz  \right\rangle.
\end{equation}

1. Let us first examine the general case,  $B$  $\beta$-cocoercive.  From (\ref{pg14}),  it follows that
\begin{equation}\label{pg15}
0 \geq   \frac{1}{2\mu} \frac{d}{dt} \| x(t)  -z  \|^2  + \frac{1}{\mu} \| \dot{x}(t) \|^2   + \beta \|  B(x(t)) -Bz \|^2 + \left\langle \dot{x}(t),  B(x(t)) -Bz  \right\rangle.
\end{equation}
Let us introduce $\alpha >0$, a positive parameter. In order to estimate the last term in (\ref{pg15}), we use Cauchy-Schwarz inequality, and the following  elementary inequality
\begin{equation}\label{pg150}
\| \dot{x}(t)\|  \|B(x(t)) -Bz \| \leq \frac{1}{2\alpha}\| \dot{x}(t)\|^2  + \frac{\alpha}{2}\|  B(x(t)) -Bz \|^2 .
\end{equation}
From (\ref{pg15}) and (\ref{pg150}) we deduce that
\begin{equation}\label{pg151}
0 \geq   \frac{1}{2\mu} \frac{d}{dt} \| x(t)  -z  \|^2  + (\frac{1}{\mu}- \frac{1}{2\alpha}) \| \dot{x}(t) \|^2  + (\beta -\frac{\alpha}{2})\|  B(x(t)) -Bz \|^2 .
\end{equation}
Choose $\alpha$ such that $\frac{1}{\mu}- \frac{1}{2\alpha} >0$ and $\beta -\frac{\alpha}{2} >0$. This is equivalent to find $\frac{\mu}{2} < \alpha < 2 \beta$, which is possible
iff $\mu < 4 \beta$, that's precisely our condition on parameters $\mu$ and $\beta$.
When this condition is satisfied, taking (for example) $\alpha = \frac{1}{2}(\frac{\mu}{2} + 2 \beta)= \frac{\mu}{4} + \beta$ in (\ref{pg151}), we obtain
\begin{equation}\label{pg16}
0 \geq   \frac{1}{2\mu} \frac{d}{dt} \| x(t)  -z  \|^2  + \frac{4\beta -\mu}{\mu(\mu + 4 \beta)} \| \dot{x}(t) \|^2  + \frac{1}{8}(4\beta -\mu)\|  B(x(t)) -Bz \|^2 .
\end{equation}
From (\ref{pg16}), it follows that, for any $z \in S$, $t \mapsto \| x(t)  -z  \|$  is   a decreasing function, and hence $\lim \| x(t)  -z  \|$ exists. Moreover, by integration of (\ref{pg16}), we obtain
\begin{equation}\label{pg17}
\int_0^{\infty} \| \dot{x}(t) \|^2  dt < + \infty, 
\end{equation}
\begin{equation}\label{pg18}
\int_0^{\infty} \|  B(x(t)) -Bz \|^2 dt < + \infty. 
\end{equation}
Since $B$ is Lipschitz continuous, and $\| \dot{x}(t) \|$ belongs to $L^2(0,+\infty)$, we have $\frac{d}{dt}B(x) \in L^2(0,+\infty)$.
Hence, by (\ref{pg18}),   $B(x) -Bz$ and its derivative belong to $L^2(0,+\infty)$. By Lemma \ref{lm:aux} we infer
\begin{equation}\label{pg19}
\lim_{t\to +\infty}B(x(t))= Bz
\end{equation}
where $Bz$ is uniquely defined for $z \in S$.
On the other hand, by  (\ref{pg00}) and (\ref{pg17}), $\dot{x}$ and its derivative belong to $L^2(0,+\infty)$. By Lemma \ref{lm:aux} we infer
\begin{equation}\label{pg20}
\lim_{t\to +\infty}\dot{x}(t)   = 0.
\end{equation}
By Opial lemma \ref{Opial},  in order to obtain the weak convergence of the orbit $x$, we just need to prove that any weak sequential cluster point of $x$ belongs to $S$.
Let $\bar{x}$ be a weak sequential cluster point of $x$, i.e.,
$\bar{x}=w-\lim x(t_{n})$ for some sequence $t_{n}\rightarrow+\infty$.
In order to pass to the limit on (\ref{pg00}), we rewrite it as 
\begin{equation}\label{pg21}
 x(t) - \mu B(x(t))   - \mbox{{\rm prox}}_{\mu \Phi}  \left( x(t) - \mu B(x(t))  \right) =  - \dot{x}(t)  - \mu B(x(t)),
\end{equation}
and use the demiclosedness property of the maximal monotone operator $I- \mbox{{\rm prox}}_{\mu \Phi} $.
Since  $x(t_n) - \mu B(x(t_n)) \rightharpoonup \bar{x} - \mu Bz$, and $\dot{x}(t) + \mu B(x(t)) \to \mu Bz$ strongly, we obtain
\begin{equation}\label{pg22}
 \bar{x} - \mu Bz   - \mbox{{\rm prox}}_{\mu \Phi}  \left( \bar{x} - \mu Bz \right) =   - \mu Bz.
\end{equation}
Equivalently,
\begin{equation}\label{pg23}
 \bar{x} = \mbox{{\rm prox}}_{\mu \Phi}  \left( \bar{x} - \mu Bz \right),
\end{equation}
that is 
$$
\partial \Phi ( \bar{x}) + Bz \ni 0.
$$
Since $B$ is maximal monotone, it is demiclosed, and hence $B \bar{x} = Bz$. Thus 
$$
\partial \Phi ( \bar{x}) + B \bar{x} \ni 0,
$$
and $\bar{x} \in S$, which completes the proof.

\smallskip

2. Now consider  $B= \nabla \Psi$, where $\Psi$ is a convex differentiable function (a special case of  cocoercive operator), and show that we can conclude to the same convergence results, without making any restrictive assumption on $\mu >0$. Let us return to (\ref{pg15}). 
By using 
$$\left\langle \dot{x}(t),  \nabla \Psi (x(t)) -  \nabla \Psi (z) \right\rangle =  \frac{d}{dt} [ \Psi(x(t)) - \left\langle \nabla \Psi (z), x(t)  \right\rangle ],$$ 
we can rewrite (\ref{pg15})  as    
\begin{equation}\label{pg24}
0 \geq   \frac{d}{dt} [\frac{1}{2\mu} \| x(t)  -z  \|^2  +\Psi(x(t)) - \Psi(z) - \left\langle \nabla \Psi (z), x(t)-z  \right\rangle ] + \frac{1}{\mu} \| \dot{x}(t) \|^2   + \beta \|  B(x(t)) -Bz \|^2 .
\end{equation}
Since $\Psi(x(t)) - \Psi(z) - \left\langle \nabla \Psi (z), x(t)-z  \right\rangle$ is nonnegative, by a similar argument as before we obtain that
 $B(x(\cdot))$ converges strongly to $Bz$    where $Bz$ is uniquely defined for $z \in S$,  $\lim_{t\to\infty}\dot{x}(t)   = 0$, and $ \| \dot{x} \| \in L^2 (0, + \infty)$.
Moreover, any  weak sequential cluster point of $x$ belongs to $S$.
But unlike the previous situation, we do not know if the  limit of $\| x(t)  -z  \|$ exists.
Instead we have $ \lim E(t,z)$ exists for any $z \in S$, where
\begin{equation}\label{pg25}
E(t,z) := \frac{1}{2\mu} \| x(t)  -z  \|^2  +\Psi(x(t)) - \Psi(z) - \left\langle \nabla \Psi (z), x(t)-z  \right\rangle .
\end{equation}
Following the arguments in \cite{Bo}, we will show that this implies that $x$ has a unique weak sequential cluster point, 
which  clearly implies the weak convergence of the whole sequence.
Let $z_1$ and $z_2$ two weak sequential cluster points of $x$, i.e., $z_1 = w-\lim x(t_n)$, and $z_2 = w-\lim x(t'_n)$, for some sequences $t_n \to +\infty $ and $t'_n \to +\infty $.
We already obtained that $z_1$ and $z_2$ belong to $S$. Hence $E(t,z_1)$ and $E(t,z_2)$ converge as $t \to + \infty$, as well as $E(t,z_1)-E(t,z_2)$.
We deduce that the following limit exists
$$
\lim_{t \to + \infty}[ \frac{1}{\mu}\left\langle  x(t), z_2-z_1  \right\rangle + \left\langle \nabla \Psi (z_2) - \nabla \Psi (z_1), x(t)  \right\rangle ]. 
$$
Thus the limits obtained by successively replacing  $t$ by $t_n$ and $t'_n$ are equal, which gives
$$
 \frac{1}{\mu}\left\langle  z_1, z_2-z_1  \right\rangle + \left\langle \nabla \Psi (z_2) - \nabla \Psi (z_1), z_1  \right\rangle =  \frac{1}{\mu}\left\langle  z_2, z_2-z_1  \right\rangle + \left\langle \nabla \Psi (z_2) - \nabla \Psi (z_1), z_2  \right\rangle. 
$$
Equivalently
$$
\frac{1}{\mu}\| z_2-z_1 \|^2   +   \left\langle \nabla \Psi (z_2) - \nabla \Psi (z_1), z_2 -z_1 \right\rangle =0,
$$
which, by monotonicity of $\nabla \Psi$, gives $z_1= z_2$.
\end{proof}

\begin{remark} {\rm Under the more restrictive assumption, $0 < \mu < 2 \beta$,  using  the results of section \ref{cocoer}, 
 the operator that governs the dynamical system  is of the form $I-T$, where $T$ is a contraction. Accordingly,
the operator $I-T$ is demipositive,  and  the weak convergence of $x$ is a direct consequence of  Bruck Theorem \cite{Bruck}.
It is an open question  whether the convergence property is true for a general cocoercive operator $B$,  without  restriction on $\mu >0$.
}
\end{remark}

\subsection{Relaxed forward-backward algorithms}

The explicit discretization of the regular dynamic (\ref{pg00})
with respect to the time variable $t$, with constant step size $h>0$,
gives 
\begin{equation}
\frac{x_{k+1}-x_{k}}{h}+ x_{k}    - \mbox{{\rm prox}}_{\mu \Phi}  \left( x_{k} - \mu B(x_{k})  \right) =0.
\end{equation}
Equivalently
\vspace{1mm}
\begin{center}
 $x_{k+1} =  (1- h)x_k +  h \mbox{{\rm prox}}_{\mu \Phi}\left( x_k  - \mu B(x_{k})\right).$
\end{center}
This is the relaxed forward-backward algorithm, whose convergence properties are well known.
The weak convergence of $(x_k)$  to an element of $S$ is obtained under the stepsize limitation: $0 < \mu < 2\beta$, and $0< h \leq 1$.
One can consult \cite[Theorem 25.8]{BC}, for the proof, and some further extensions of this result.

\section{Perspective}

Our work can be considered from two perspectives: numerical splitting methods in optimization, and modeling in physics, decision sciences.

1. In recent years, there has been a great interest in the forward-backward methods, especially in the signal/image processing, and sparse optimization. A better understanding of these methods is a key to obtain further developments, and improvement of the methods:
fast converging algorithms, nonconvex setting, multiobjective optimization, ... are crucial points to consider in the future.
To cite some of these topics, in \cite{ABS} the convergence of the classical forward-backward method, in a nonconvex nonsmooth framework, has been proved for functions satisfying the Kurdyka-Lojasiewiz inequality, a large class  containing the semi-algebraic functions. 
The proof finds its roots in a dynamical argument. It is an open question to know if some other form of the FB algorithms works in this setting. Even for the relaxed FB algorithm this is an open question. \\ 
Similarly, the Nesterov method for obtaining   convergence rate $O(\frac{1}{k^2})$, is known for the classical forward-backward algorithm, see \cite{Nest2}, \cite{BT09} (FISTA method). 
It would be very interesting to know if the method can be adapted to other forms of these algorithms.\\
It turns out that there is a rich family of forward-backward algorithms. In this article, we have considered three classes of these algorithms.
The link with the dynamical systems is a valuable tool for studying  these algorithms, and to discover new one.
The comparison between the algorithms that are obtained by time discretization of the continuous dynamics is a delicate subject, which is the subject of current research.

\bigskip

2. Many equilibrium problems in physical sciences or decision may be written either as  convex minimization problem or as a search for a fixed point of a contraction.
Often these two aspects are present simultaneously. For example, in game theory, agents may adopt strategies involving  cooperative aspects (potential games) and noncooperative aspects.
Nash equilibrium formulation can lead to a convex-concave saddle value problem, and non-potential monotone operators.
An abundant literature has been devoted to finding common solutions of these problems.
In contrast, our approach aims at finding a compromise solution of these two different types of problems. A basic ingredient is the resolution of $ \partial \Phi (x) + Bx \ni 0$.
An interesting direction for future research would be to consider a multicriteria dynamical process associated to the  two operators $ \partial \Phi$ and  $B$, in  line of the recent article 
\cite{AG}. \\
The selection of equilibria with desirable properties is an important issue in decision sciences.
  With the introduction of regularization terms tending asymptotically to zero, not too quickly (eg Tikhonov type),  the dynamic equilibrium approach provides an asymptotic hierarchical selection. 
There is an extensive literature on this topic, see  \cite{atcom2}, \cite{AttCza}, \cite{BaCo}, \cite {Wei Bian}, \cite{Bo}, 
\cite{Cab}, \cite{CH}, \cite{CPS}, \cite{Hir}, and references therein. It is an issue that is largely unexplored for the systems considered in this article.

\end{document}